\documentclass[12pt,reqno]{amsart}

\usepackage[english]{babel}

\usepackage[letterpaper,hmargin=2.5cm,vmargin=2.5cm]{geometry}
\usepackage{ws-rotating}     
\usepackage{graphicx}
\usepackage{epstopdf}
\usepackage{amsfonts,amssymb,amsmath}
\usepackage{subfigure} 
\usepackage[utf8]{inputenc}
\usepackage{tikz}
\usepackage{xcolor}

\usepackage[dvipsnames]{xcolor}
\usepackage[
    colorlinks=true,
    linkcolor=MidnightBlue,   
    citecolor=MidnightBlue,   
    urlcolor=RoyalBlue        
]{hyperref}

\newtheorem{corollary}{Corollary}

\newtheorem{proposition}{Proposition}

\graphicspath{{./figs/}} 
\usepackage{url}

\begin{document}

\title[Bifurcations in a Saltzman-Maasch model]{Hopf-Bautin and homoclinic bifurcations in a  Saltzman-Maasch model with cubic feedback}

\author[M. P. Garc\'{\i}a-Rivera et al.]{Marco Polo Garc\'{\i}a-Rivera, Martha Alvarez-Ram\'{\i}rez}
\address{Departamento de Matem\'aticas, UAM--Iztapalapa,\\
09310 Iztapalapa, Mexico City, Mexico}
\email{mpgr86@gmail.com, mar@xanum.uam.mx}

\author[]{Hildeberto Jard\'on-Kojakhmetov}
\address{Johann Bernoulli Institute for Mathematics and Computer Science,\\ 
University of Groningen, P.O. Box 407, 9700 AK\\ Groningen, The Netherlands}
\email{h.jardon.kojakhmetov@rug.nl}

\maketitle

\begin{abstract}
This paper investigates a deterministic variant of the Saltzman-Maasch model for Pleistocene glacial cycles, formulated as a three-dimensional dynamical system with cubic feedback in the atmospheric carbon dioxide equation. After reducing the model to a planar system on a critical manifold, we perform a detailed bifurcation analysis and analytically identify both Hopf and Bautin (generalized Hopf) bifurcations, which govern the emergence of stable and unstable limit cycles. To analyze global transitions, we perform a rescaling of time and variables to derive a leading-order Hamiltonian system. This reduction enables the explicit construction of homoclinic orbits and the application of Melnikov’s method to assess their persistence under perturbations. 
The analytical predictions are further validated through numerical continuation and simulations, providing a rigorous foundation for previously reported numerical observations and establishing, in particular, the analytical existence of Bautin bifurcations, homoclinic connections via Melnikov analysis, and a systematic slow--fast geometric reduction of the model.
\end{abstract}


\section{Introduction}
Pleistocene climate variability is characterized by alternating glacial and interglacial periods, commonly attributed to variations in Earth's orbital parameters (Milankovitch forcing). Although these forcings are quasiperiodic and relatively weak, the climate system exhibits nonlinear and sometimes abrupt responses, suggesting the presence of internal amplification mechanisms. 

Conceptual climate models provide a useful framework for investigating such
mechanisms. Among them, the Saltzman-Maasch (SM) model
\cite{Saltzman_Maasch_1990, Maasch1990} describes the coupled evolution of
global ice mass, atmospheric CO$_2$, and deep-ocean circulation through a
low-dimensional system of ordinary differential equations. Recent work such as \cite{Shen2026b}, has emphasized the role of nonlinear
feedbacks, bistability, and tipping phenomena in such models; in particular,
Shen considers a conceptual energy balance model in which a cubic nonlinear
term, arising from ice-albedo feedback, introduces an effective negative
feedback that limits linear instability and produces a double-well potential
structure supporting multiple stable climatic regimes and abrupt transitions
between them. Here, we analyze a Saltzman-Maasch type model with cubic
feedback from the perspective of local and global bifurcation theory,
establishing analytically the occurrence of Hopf, generalized Hopf (Bautin),
and homoclinic bifurcations.
In nondimensional form, the original model \cite{Saltzman_Maasch_1990, Maasch1990} reads as:
\begin{equation}
\label{eq_hans}
\begin{array}{l}
\dot{x} = -x - y,\\
\dot{y} = ry -pz  + sz^2 - yz^2,\\
\dot{z} = -q(x + z),
\end{array}
\end{equation}
where $x, y,$ and $z$ represent anomalies in global ice mass, atmospheric CO$_2$ concentration, and deep-ocean circulation strength, respectively, and $p, q, r, s > 0$ are dimensionless parameters encoding effective physical rates. This model has been widely used to study glacial--interglacial dynamics and the role of internal feedback mechanisms; in particular, the nonlinear terms in the CO$_2$ equation are central to shaping the oscillatory behavior of the system.

The parameter $q$ represents the ratio of characteristic time scales between the deep-ocean circulation and the slower climatic variables (ice mass and atmospheric CO$_2$). The regime $q \gg 1$ therefore corresponds to a separation of time scales in which the ocean adjusts much faster than the other components, naturally inducing a slow--fast dynamical structure. System~\eqref{eq_hans} has been extensively analyzed in the literature; notably, Engler et al. \cite{Engler2017} studied this model in the singular perturbation regime $q \gg 1$, showing that the dynamics are organized by invariant slow or center manifolds and structured by Bogdanov--Takens bifurcations.

Motivated by nonlinear feedback mechanisms in the carbon cycle, we consider a modification of system~\eqref{eq_hans} in which the atmospheric CO$_2$ equation includes additional nonlinear self-interaction terms. This leads to the cubic model:
\begin{equation}
\label{MSN}
\begin{array}{l}
\dot{x} = -x - y,\\
\dot{y} = ry -pz  - sy^2 - y^3,\\
\dot{z} = -q(x + z).
\end{array}
\end{equation}
The cubic nonlinearity in the CO$_2$ equation (the $y$-component) can be interpreted as a low-order approximation of nonlinear atmosphere--ocean carbon exchange processes. In conceptual climate models, polynomial feedback terms are commonly used to represent nonlinear saturation and self-limiting effects in the carbon cycle. In particular, cubic terms provide the simplest extension beyond quadratic interactions, allowing for richer dynamical behavior while preserving a minimal model structure.
As in \cite{Nolet2017}, we consider the cubic model~\eqref{MSN} in the regime $q \gg 1$, so that it also exhibits a natural slow--fast structure.

While Engler et al. \cite{Engler2017} provided a rigorous geometric singular perturbation analysis for the  SM model \eqref{eq_hans}, establishing a solid analytical foundation for its slow--fast dynamics, a similar treatment for the cubic model \eqref{MSN} is currently lacking. Although \cite{Nolet2017} explored system \eqref{MSN} and identified Hopf, Bogdanov--Takens, and generalized Hopf (Bautin) bifurcations through numerical continuation, these findings have not yet been supported by a formal analytical framework. In particular, the bifurcations in the cubic formulation have not been established through rigorous stability theory, nor has its slow--fast structure been systematically developed within a geometric singular perturbation setting.

The present work addresses these gaps by providing a unified and fully analytical framework for the cubic model~\eqref{MSN}. In particular, we rigorously establish the existence of Bautin (generalized Hopf) bifurcations, explicitly construct homoclinic orbits and analyze their persistence using Melnikov theory, and develop a systematic slow--fast reduction based on geometric singular perturbation theory. 
Together, these results extend the geometric understanding previously available for system~\eqref{eq_hans} to the richer bifurcation structure of the cubic model~\eqref{MSN}, providing a coherent analytical explanation of both local and global dynamical mechanisms.

In a related line of research, Quinn \cite{Quinn2018} explored system \eqref{MSN} to study delayed-feedback mechanisms, while stochastic extensions of the SM model have been addressed by Alexandrov \cite{Alexandrov2024} and Roberts \cite{Roberts2015}. In contrast, the present study focuses on deterministic mechanisms and their geometric organization.

We organize the paper as follows. In Section \ref{sec:model}, we set up the slow-fast model and derive the reduced dynamics on the critical manifold. Section \ref{sec_pts} outlines the conditions for the existence of equilibria and examines their local stability. In Section \ref{sec_hopf}, we analyze Hopf bifurcations using $r, s,$ and $p$ as parameters, compute the associated Lyapunov coefficients, and characterize the generalized Hopf (Bautin) bifurcation. 
In addition, we identify Bogdanov--Takens bifurcations that organize the dynamics. Section \ref{sec_hamilton} is devoted to their unfolding near Bogdanov--Takens points and to the derivation of a Hamiltonian limit, where homoclinic orbits are explicitly constructed and their persistence is determined via Melnikov theory. In this context, we distinguish between the symmetric case $s=0$, which is included for mathematical completeness, and the physically meaningful case $s>0$, where the nonlinear self-interaction in the atmospheric CO$_2$ dynamics breaks the symmetry.
Section \ref{sub_sn0} is devoted to the asymmetric case $s>0$ and its effect on the homoclinic structure, while Section \ref{sub_s0} focuses on the symmetric case $s=0$. Finally, we conclude in Section \ref{sec_conclusion}.
To improve the flow and readability of the manuscript, detailed and lengthy calculations of the Lyapunov coefficients and Melnikov integrals are presented in the appendices.

\section{The slow--fast model and critical manifold analysis}\label{sec:model}
In this section, we reformulate system~\eqref{MSN} within a slow-fast framework, 
corresponding to the physically relevant regime in which the North Atlantic Deep Water (NADW) 
adjusts on a much faster timescale than the total ice mass. This assumption is consistent with 
previous studies of the Maasch--Saltzman model (see, e.g.,~\cite{Engler2017}), although the model 
considered here includes additional nonlinear feedback mechanisms.

Accordingly, we focus on the parameter regime $q \gg 1$, and write $q = 1/\varepsilon$ with $0 < \varepsilon \ll 1$. The variables $x$, $y$, and $z$ represent the anomalies of total ice mass, atmospheric CO$_2$ concentration, and deep-ocean temperature/salinity, respectively, while $p$, $r$, and $s$ are positive parameters. Since $q$ measures the ratio of time scales, the regime $q \gg 1$ implies that
$z$ evolves faster than $x$ and $y$. Substituting $q = 1/\varepsilon$ into
system~\eqref{MSN} yields the slow--fast system
\begin{equation}\label{MSNe}
\begin{aligned}
    \dot{x} &= -x - y,\\
    \dot{y} &= ry - pz - sy^2 - y^3,\\
    \varepsilon\dot{z} &= -(x + z),
\end{aligned}
\end{equation}
where the overdot denotes differentiation with respect to the \emph{slow time}
$t$. Rewriting system~\eqref{MSNe} in the fast time $\tau := t/\varepsilon$
gives
\begin{equation}\label{MSNf}
\begin{aligned}
    x' &= -\varepsilon(x + y),\\
    y' &= \varepsilon(ry - pz - sy^2 - y^3),\\
    z' &= -(x + z),
\end{aligned}
\end{equation}
where $'$ denotes differentiation with respect to $\tau$.

We now analyze the dynamics of system~\eqref{MSNe} on the critical manifold
arising in the singular limit $\varepsilon \to 0$, which reveals a range of
phenomena beyond those previously reported. Setting $\varepsilon = 0$ in
\eqref{MSNe} and~\eqref{MSNf} yields, respectively, the  \emph{reduced}  and the \emph{layer} equations:
\begin{equation}\label{MSN:reduced_layer}
\begin{alignedat}{4}
    \dot{x} &= -x - y,       &\qquad\qquad  x' &= 0,          \\[2mm]
    \dot{y} &= ry - pz - sy^{2} - y^{3},   &  y' &= 0,          \\[2mm]
           0 &= -(x + z),    &              z' &= -(x + z).
\end{alignedat}
\end{equation}

The equilibria of the layer system define the
critical manifold
$$\mathcal{M}_0 = \bigl\{(x,y,z)\in\mathbb{R}^3 \mid z = -x\bigr\},$$
which serves simultaneously as the phase space of the reduced equations and
as the geometric object governing the slow dynamics of
system~\eqref{MSNe} in the singular limit $\varepsilon \to 0$. Normal
hyperbolicity of $\mathcal{M}_0$ follows immediately from the fact that the
linearization of the layer equations transverse to $\mathcal{M}_0$ has the
single eigenvalue $\lambda = -1 < 0$, so $\mathcal{M}_0$ is normally
hyperbolic and attracting on its entire extent. Substituting the constraint
$z = -x$ into the differential equations of the reduced
system yields the two-dimensional \emph{reduced
problem}
\begin{equation}
\label{MSNcero}
\begin{aligned}
    \dot{x} &= -x - y, \\
    \dot{y} &= px + ry - sy^2 - y^3,
\end{aligned}
\end{equation}
which completely describes the slow flow on $\mathcal{M}_0$.

According to geometric singular perturbation theory (GSPT for short) of  Fenichel \cite{Fenichel}, the normal hyperbolicity of $\mathcal{M}_0$ ensures the existence of a nearby slow manifold $\mathcal{M}_\varepsilon$ for any sufficiently small $\varepsilon >0$. This manifold can be expressed as a graph:
\begin{equation}
\label{eqinv}
\mathcal{M}_{\varepsilon}=\: \{ (x,y,z)\: |\: z=h_{\varepsilon}(x,y)\},
\end{equation}
and it attracts nearby trajectories exponentially fast in the fast variable. The flow on $\mathcal{M}_{\varepsilon}$ 
converges to the dynamics on $\mathcal{M}_0$ as $\varepsilon\to0$.

In the remainder of this section, we analyze the reduced system~\eqref{MSNcero}, focusing on the existence and stability of its equilibria and the bifurcations they undergo. This provides a foundation for the study of oscillatory behavior emerging in the full system, including the role of Hopf,  generalized Hopf and homoclinic  bifurcations in organizing the model’s long-term dynamics.
Throughout this analysis, the existence of equilibria is obtained from the associated algebraic equations, while all stability conditions are derived directly from the Jacobian matrix and have been systematically verified using symbolic computation, ensuring full consistency between the existence and stability regions.

\subsection{Existence of equilibria}\label{sec_pts}
To determine the equilibrium points of the reduced system~\eqref{MSNcero}, we set the right-hand sides of both equations equal to zero. The origin $(0,0)$ is always an equilibrium point.
Any additional equilibria must satisfy $y=-x$. Substituting this relation into system~\eqref{MSNcero}, we obtain the quadratic equation
\begin{equation}\label{eq:quad}
x^2 - s x + p - r = 0.
\end{equation}
Thus, the equilibria are given by the origin together with the real solutions of \eqref{eq:quad}.
Let
\begin{equation}\label{eq:discr}
    \Delta = s^2 + 4(r - p)
\end{equation}

denote the discriminant of \eqref{eq:quad}. The number of equilibria depends on the sign of $\Delta$.
\begin{proposition}\label{pprop1}
Assume $p>0$, $r>0$, and $s \ge 0$. Then, the origin $P_0:=(0,0)$ is always an equilibrium of system \eqref{MSNcero}. Additional equilibria lie on the line $y=-x$ and are determined by \eqref{eq:quad} and its discriminant \eqref{eq:discr}. To be precise:
\begin{enumerate}
\item[{1.}] If $\Delta < 0$, then the origin, $P_0$, is the only equilibrium.
\item[{2.}] If $\Delta = 0$, then additional to $P_0$, there exists a degenerate equilibrium at $(x,y)=\left(\frac{s}{2}, -\frac{s}{2}\right)$.
\item[{3.}] If $\Delta > 0$, then \eqref{eq:quad} has two real roots
\begin{equation}\label{eq:equs}
x_{1,2}^* = \frac{1}{2}\left(s \pm \sqrt{s^2 + 4(r - p)}\right),
\end{equation}
which yield two equilibria
\[
P_1:=(x_1^*,-x_1^*), \qquad P_2:=(x_2^*,-x_2^*).
\]
If $p=r$, then $x_2^*=0$, so $P_2=(0,0)$ coincides with the origin, and system~\eqref{MSNcero} has exactly two distinct equilibria: the origin $P_0$ and $P_1=(s,-s)$.
\end{enumerate}
\end{proposition}

\noindent To avoid ambiguity, we explicitly analyze the degenerate case $s=0$ in the next result.
\begin{corollary}
When $s=0$, equation~\eqref{eq:quad} reduces to $x^2 + (p-r)=0$, and the cases of 
Proposition~\ref{pprop1} simplify as follows:
\begin{enumerate}
\item[{1.}] If $r < p$, then $\Delta < 0$ and the origin is the unique equilibrium.
\item[{2.}] If $r = p$, then $\Delta = 0$ and the origin is again the unique equilibrium.
\item[{3.}] If $r > p$, then $\Delta > 0$ and system~\eqref{MSNcero} has exactly three 
equilibria: the origin $P_0$ and the two symmetric points 
$P_{1,2} = \left(\pm\sqrt{r-p},\, \mp\sqrt{r-p}\right)$.
\end{enumerate}
\end{corollary}

The next subsection examines the local stability of the equilibria by computing the eigenvalues of the Jacobian matrix at each equilibrium point. In particular, we identify conditions under which the Jacobian admits a pair of purely imaginary eigenvalues, a necessary condition for the onset of oscillatory dynamics near equilibrium (see, for instance, \cite{Guckenheimer}).

\subsection{Local stability of the origin}
We investigate the stability of the trivial equilibrium point $(0,0)$ by linearizing the system~\eqref{MSNcero} around this point. The corresponding Jacobian matrix takes the form:
\begin{equation}
\label{jacobiana2}
\begin{pmatrix}
-1 & -1\\
p & r\\
\end{pmatrix}
\end{equation}
with  characteristic equation 
$$\lambda^2+(1-r)\lambda +p-r =0.$$
The associated eigenvalues are 
\begin{equation}\label{val_cero}
\lambda_{1,2} = \frac{1}{2}\left( r-1\pm \sqrt{(1+r)^2-4p}\right).
\end{equation}
To characterize the local behavior near $(0,0)$, we examine the sign and nature of these eigenvalues. This leads to the following classification result.
\begin{proposition} \label{prop_stab}
Let $r>0$ and $p>0$. The stability of the  equilibrium point $(0,0)$ of system \eqref{MSNcero} is determined as follows:
\begin{enumerate}
\item[1.] if $r>1$ and $r<p<\frac{1}{4}(r+1)^2$:  an unstable node.
\item[2.] if $r>1$ and $p=\frac{1}{4}(r+1)^2$:  a degenerate unstable node.
\item[3.] if $r>1$ and $p>\frac{1}{4}(r+1)^2$:   an unstable focus.
\item[4.] if $r<1$ and $p>\frac{1}{4}(r+1)^2$:   a stable focus.
\item[5.] if $r<1$ and $p=\frac{1}{4}(r+1)^2$:   a degenerate stable node.
\item[6.] if $r<1$ and $r<p<\frac{1}{4}(r+1)^2$: a stable node.
\item[7.] if $r=1$ and $p>1$:  a linear center. 
\item[8.] if $0< p<r$:  a saddle point. 
\end{enumerate}
\end{proposition}

Proof. If $p-r<0$, then $\det(J)<0$, so the eigenvalues are real and of opposite signs. Hence, the origin is a saddle point.
Assume now that $p-r>0$. If
\[
r<p<\frac{(r+1)^2}{4},
\]
then the discriminant is positive and both eigenvalues are real. In this case, their sum is
\[
\lambda_1+\lambda_2=r-1.
\]
Therefore, if $r>1$, then $\lambda_1,\lambda_2>0$, and the origin is an unstable node. If $r<1$, then $\lambda_1,\lambda_2<0$, and the origin is a stable node.
At $p=\frac{(r+1)^2}{4}$, the eigenvalues coincide and are given by
\[
\lambda_1=\lambda_2=\frac{r-1}{2}.
\]
Hence, the origin is a degenerate node, which is unstable if $r>1$ and stable if $r<1$.
Finally, if
\[
p>\frac{(r+1)^2}{4},
\]
then the eigenvalues are complex conjugates with real part
\[
\mbox{Re}\:(\lambda_{1,2})=\frac{r-1}{2}.
\]
Hence, if $r>1$, the origin is an unstable focus; if $r<1$, it is a stable focus. When $r=1$ and $p>1$, the eigenvalues are purely imaginary, so the origin is a linear center.

This result describes the stability of the origin. In particular, the change at $r=1$
marks the transition from a stable equilibrium to an unstable one and identifies the parameter regime where a Hopf bifurcation may occur.

\subsection{Stability analysis of equilibria outside the origin}
Let $P_i=(x_i^*,-x_i^*)$, $i=1,2$ (see \eqref{eq:equs}), denote the two additional equilibrium points. 
To analyse their linear stability, we consider the Jacobian matrix of system \eqref{MSNcero} at $P_i$, which yields the Jacobian matrix
\begin{equation}\label{matjp}
J_i=\begin{pmatrix}
-1 & -1\\
p & r+2 s x_i^*-3(x_i^*)^2
\end{pmatrix}.
\end{equation}

The eigenvalues of $J_i$ are determined by the characteristic equation
\[
\lambda^2 - \operatorname{tr}(J_i)\lambda + \det(J_i)=0,
\]
where
\[
\operatorname{tr}(J_i) = r-1+2sx_i^*-3(x_i^*)^2, 
\qquad
\det(J_i) = p-r-2sx_i^*+3(x_i^*)^2
\]
denote the trace and determinant of $J_i$, respectively.

According to the Routh--Hurwitz criterion, both eigenvalues have negative real parts if and only if
\[
\operatorname{tr}(J_i) < 0 \quad \text{and} \quad \det(J_i) > 0,
\]
see \cite{Dumortier}. Furthermore, if $\operatorname{tr}(J_i) = 0$ and $\det(J_i) > 0$, then the eigenvalues form a pair of purely imaginary complex conjugates.

\medskip

In the case $s=0$, we have $(x_1^*)^2 = r - p$, and therefore $x_1^* = \pm \sqrt{r - p}$. Substituting into the expressions for the trace and determinant of the Jacobian matrix yields
\[
\det(J_1)=2(r-p), \qquad \operatorname{tr}(J_1)=-2r-1+3p.
\]
Hence, the conditions $\det(J_1)>0$ and $\operatorname{tr}(J_1)<0$ are equivalent to $r>p$ and $p<\frac{1+2r}{3}$, respectively. Combining these inequalities, we obtain
\[
0<p<\min\left\{r,\tfrac{1+2r}{3}\right\}.
\]

\medskip

Applying the Routh--Hurwitz criterion to the Jacobian matrix at $P_1$, we now derive the stability conditions for $s>0$. These are obtained by substituting the explicit expressions of $x_i^*$ from \eqref{eq:equs} into $\operatorname{tr}(J_i)$ and $\det(J_i)$ and solving the inequalities $\operatorname{tr}(J_i)<0$ and $\det(J_i)>0$.

For $s>0$, substituting
\[
x_1^*=\frac{1}{2}\left(s+\sqrt{s^2+4(r-p)}\right)
\]
into the trace gives
\[
\operatorname{tr}(J_1)
=
3p-2r-1-\frac{s^2}{2}
-\frac{s}{2}\sqrt{s^2+4(r-p)}.
\]
Thus, the condition $\operatorname{tr}(J_1)<0$ is equivalent to
\[
3p-2r-1-\frac{s^2}{2}
-\frac{s}{2}\sqrt{s^2+4(r-p)}<0.
\]
Solving the equation $\operatorname{tr}(J_1)=0$ for $r$ yields the Hopf boundary
\[
r=
\frac{1}{8}\left(-4+12p-s^2-s\sqrt{s^2+8p-8}\right).
\]
Therefore, the inequality $\operatorname{tr}(J_1)<0$ implies
\[
r>\frac{1}{8}\left(-4+12p-s^2-s\sqrt{s^2+8p-8}\right).
\]

Collecting the above results, we obtain the following characterization of the stability of $P_1$ and $P_2$.
\begin{proposition}\label{prop3}
The following statements hold for system \eqref{MSNcero}. The equilibrium point $P_1$ is stable:
\begin{enumerate}
\item[1.] For $s=0$, $r>0$, and
\[
0<p<\min\left\{r,\frac{1+2r}{3}\right\}.
\]
\item[2.] For $s>0$, $0<p\leq 1$, and $p<r$.
\item[3.] For $s>0$, $p>1$, and
\[
r>\tfrac{1}{8}\left(-4+12p-s^2-s\sqrt{s^2+8p-8}\right).
\]
\end{enumerate}
\end{proposition}

\begin{proposition}\label{prop4}
The following statements hold for system \eqref{MSNcero}. The equilibrium point $P_2$ is stable:
\begin{enumerate}
\item[1.] For $s>0$, $0<p\leq 1$, and $p<r$.
\item[2.] For $s>0$, $p>1$, and
\[
r>\frac{1}{8}\left(-4+12p-s^2+s\sqrt{s^2+8p-8}\right).
\]
\end{enumerate}
\end{proposition}
\medskip

The stability conditions in Propositions~\ref{prop3} and \ref{prop4} define the regions in parameter space where $\operatorname{tr}(J_i) < 0$, while the corresponding equalities characterize the Hopf bifurcation curves described in the next proposition. From a dynamical viewpoint, Proposition~\ref{prop_stab} characterizes the stability of the origin and identifies the parameter regime where a Hopf bifurcation may occur, while Propositions~\ref{prop3} and~\ref{prop4} describe the stability of the nontrivial equilibria $P_1$ and $P_2$, which organize the system dynamics and possible transitions between regimes.

\begin{proposition}
The differential system \eqref{MSNcero} possesses a pair of purely imaginary eigenvalues $\pm i \omega$, with $\omega >0$, under the following conditions.
\begin{enumerate}
\item[1.] At the origin  for  $s> 0$, $p>1$,  and $r=1$.
\item[2.] At $P_1$ for $s> 0$, $p>1$, and  $r= \frac{1}{8} (-4+12p-s^2-s\sqrt{s^2+8p-8})$.
\item[3.] At $P_2$ for $s> 0$, $p>1$, and $r= \frac{1}{8} (-4+12p-s^2+s\sqrt{s^2+8p-8})$.
\end{enumerate}
\end{proposition}

We point out that when $p=r+\frac{1}{4}s^2$, the equilibrium points $P_1$ and $P_2$ coalesce at $(x,y)=(\frac{s}{2},-\frac{s}{2})$, and the corresponding Jacobian matrix has a zero eigenvalue,  while the other eigenvalue is  $\lambda  = r+\frac{s^2}{4}-1\neq 0$, indicating that the equilibrium point is  non-hyperbolic. In this case, Sotomayor’s theorem (\cite{Perko}, $\S$4.2) must be applied to determine the conditions under which saddle-node and transcritical bifurcations occur.
In particular, when $p = 1$ and $r = 1 - \frac{1}{4}s^2$, the Jacobian matrix has double zero eigenvalues, signaling the presence of a codimension-two Takens-Bogdanov bifurcation  (see \cite{Perko}, $\S$4.13). A similar bifurcation also occurs at the origin $(0,0)$ when $p = r = 1$.

\section{Hopf bifurcation analysis}\label{sec_hopf}
In this section, we derive sufficient conditions for the occurrence of Hopf bifurcations and analyse the resulting oscillatory behaviour. To determine the criticality of these bifurcations, we  apply Hopf's Theorem   together with the explicit formulas for the Lyapunov coefficient given in  \cite{Kuznetsov2023}. The parameters  $r$ and $p$ are treated as bifurcation parameters.

In the context of the Maasch-Saltzman model, Hopf bifurcations play a central role, as they provide a fundamental mechanism for the emergence of glacial-interglacial cycles. These bifurcations mark the onset of self-sustained oscillations emerging from steady states, capturing the alternating pattern between colder and warmer climate phases.  The associated limit cycles reflect essential aspects of Pleistocene climate variability, shaped by nonlinear feedbacks among ice volume, atmospheric CO$_2$, and deep ocean temperature.

\subsection{Hopf bifurcation at the origin}
The following proposition characterizes the Hopf bifurcation occurring at the origin of system \eqref{MSNcero}, where $l_1$ and $l_2$ denote the first and second Lyapunov coefficients, respectively.

\begin{proposition}\label{prop_bif1}
The origin of system \eqref{MSNcero} undergoes a Hopf bifurcation at $r=1$ for $p>1$. More precisely:
\begin{enumerate}
\item[1.] $r=1$, $p>1$: Hopf bifurcation at the origin.
\item[2.] $1<p<1+\frac{2}{3}s^2$: subcritical Hopf ($l_1>0$).
\item[3.] $p>1+\frac{2}{3}s^2$: supercritical Hopf ($l_1<0$).
\item[4.] $p=1+\frac{2}{3}s^2$: codimension-2 Bautin bifurcation.
In this case, the second Lyapunov coefficient satisfies
$$l_2=-\frac{5(3+2s^2)^4}{128\sqrt{6}\,s^7}<0 \quad \text{for all } s>0.$$
\end{enumerate}
\end{proposition}

The proof of Proposition \ref{prop_bif1} is provided in Appendix \ref{ap_lyap}.

The generalized Hopf bifurcation can be verified, numerically, by setting $r = 1$, $s = 1$, and $p = \frac{13}{10}$. This is illustrated in Figure \ref{bautin}, where two limit cycles emerge from the generalized Hopf point: the outer cycle is stable, while the inner cycle is unstable.
Additionally, the time series corresponding to these limit cycles are presented in Figure \ref{series-bautin}. It is important to note that the presence of a limit cycle in the model helps explain climate fluctuations. In this context, the stable limit cycle resulting from the generalized Hopf bifurcation generates oscillations in ice mass, carbon dioxide concentration, and mean ocean temperature, with both stable and unstable behaviors.
\begin{figure}[h!]
\begin{center} 
\begin{tikzpicture}
    \node at (0,0){\includegraphics[scale=0.55]{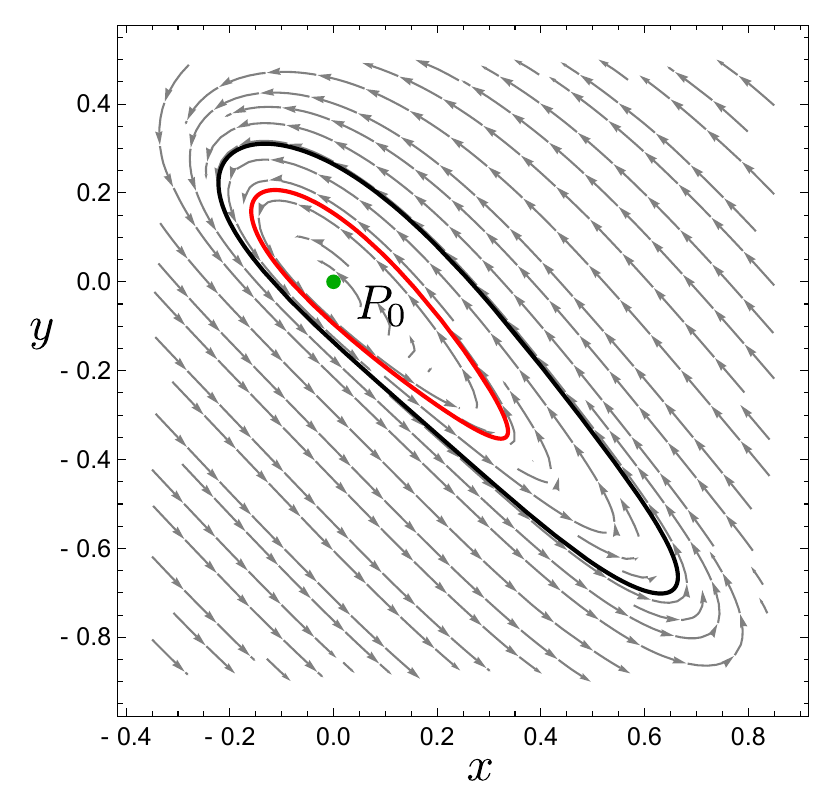} };
\end{tikzpicture}
\end{center}
\caption{The phase portrait in the $(x,y)$-plane of system \eqref{MSNcero} for $r = 0.96$, $s = 1$, and $p = 1.3$ reveals the presence of two limit cycles around the origin (in green). The inner limit cycle, shown in red, is unstable, while the outer limit cycle, depicted in black, is stable. 
}
\label{bautin}
\end{figure}

\begin{figure}[h!]
\begin{center} 
\subfigure[]{
\includegraphics[scale=0.45]{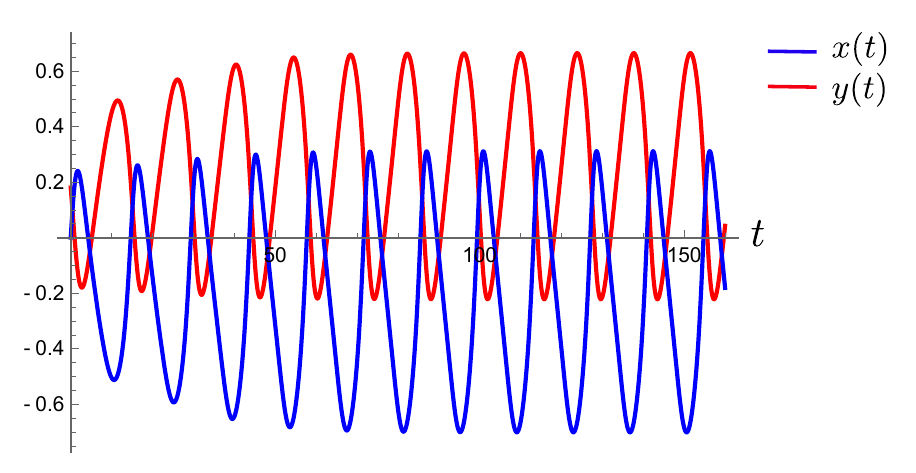}}\hspace{0.2cm}
 \subfigure[]{
\includegraphics[scale=0.45]{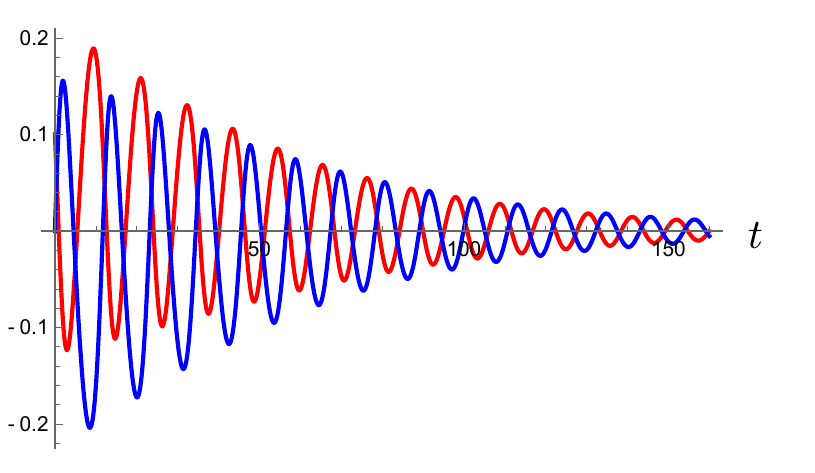}}
\end{center}
\caption{Time series of system~\eqref{MSNcero} illustrating distinct dynamical behaviors associated with stable and unstable limit cycles (see Figure~\ref{bautin}) for $p=1.3$, $r=0.96$, and $s=1$. The plots show the evolution of ice volume ($x$, blue) and atmospheric CO$_2$ concentration ($y$, red) for the initial conditions (a) $(x,y)=(0.18,0)$ and (b) $(x,y)=(0.1,0)$. In (a), the solution exhibits sustained oscillations, whereas in (b), the oscillations decay in amplitude and the trajectory converges to the equilibrium at zero.} \label{series-bautin}
\end{figure}

\section{Hopf bifurcations at nontrivial equilibria}
Before discussing the Hopf bifurcations at the nontrivial equilibria, we recall that the
variables in the nondimensional model represent anomalies (deviations from long-term mean
values) of total ice mass, atmospheric CO$_2$ concentration, and deep-ocean temperature.
Accordingly, oscillations around the origin correspond to fluctuations about the mean
climate state and do not imply physically unrealistic values of the variables. The
nontrivial equilibria should thus be understood as perturbed quasi-steady climate states
within this anomaly framework, rather than as absolute physical baselines.

In the context of the Saltzman-Maasch model \eqref{MSNcero}, Hopf bifurcations at non-origin equilibrium points $P_1$ and $P_2$ are crucial to capture the observed transitions between glacial and interglacial cycles. Unlike the origin, which typically represents a trivial or symmetric equilibrium state (e.g., an ice-free climate baseline), non-zero equilibria correspond to physically meaningful climate states with finite ice volume, 
CO$_2$ concentration, and ocean temperature. When Hopf bifurcations occur at these points, they can generate stable or unstable limit cycles, producing sustained oscillations that replicate the periodic nature of glacial-interglacial dynamics.

\begin{proposition}\label{propl1}\text{}
The equilibrium point $P_1$ undergoes a subcritical Hopf bifurcation when
 $ s\geq 0$, $p>1$,  and $r=\frac{1}{8}\left( -4+12p-s^2-s\sqrt{-8 + 8 p+s^2}\right)$.
\end{proposition}

\begin{proof}
We start by shifting $P_1$ to the origin, by the use the coordinate transformation
$$
X = x+x_1^*, \qquad Y = y-x_1^*.
$$
After the translation, the system \eqref{MSNcero} becomes:
\begin{equation}
\label{trass1}
\begin{split}
\dot{X}&=\: -X-Y,\\
\dot{Y}&=\: pX -\frac{1}{2}  \Big(-6p + 4r+s^2+s \sqrt{-4p+4r+s^2}\Big)Y  \\
 & \hspace{3cm} +  \frac{1}{2}\Big(s + 3\sqrt{-4p+4r+s^2}\Big) Y^2 -  Y^3.
\end{split}
\end{equation}

For a Hopf bifurcation to occur at the equilibrium point $(0,0)$ of system \eqref{trass1},  the Jacobian matrix must satisfy two conditions:  it must have a positive determinant and zero trace.
These conditions lead in the following expression: 
$$ r = \frac{1}{8} (-4+12p-s^2)  - \frac{s}{8} \sqrt{s^2 + 8(p-1)},$$
which holds under the constraint  $p>1$. Furthermore, the condition 
$$
\sqrt{2}\sqrt{-4+4p+s^2-s \sqrt{8(p-1) + s^2}} = -s+\sqrt{s^2 +8(p-1)}.
$$
must also be satisfied.

Using these identities, the system \eqref{trass1} can be rewritten as
\begin{equation}
\label{tras1}
\begin{split}
\dot{X}&=\: -X-Y,\\
\dot{Y}&=\: pX +Y-\tfrac{1}{4}Y^2\left(s-3\sqrt{s^2+8(p-1)}+4Y\right).
\end{split}
\end{equation}
Now, the Jacobian matrix of system \eqref{tras1} at the equilibrium point $(0,0)$ is given by
\begin{equation}
\label{jacp1}
J=\left(
\begin{array}{cc}
 -1 & -1 \\
 p & 1 \\
\end{array}
\right).
\end{equation}
This matrix has a pair of purely imaginary eigenvalues $\lambda_{1,2} = \pm i\omega$, with $\omega = \sqrt{p - 1}$, which requires $p > 1$.

We now apply Hopf bifurcation theory to assess the stability of the equilibrium point and characterize the nature of the emerging periodic orbits. This analysis is carried out through the computation of Lyapunov coefficients.

As in Proposition \ref{prop_bif1}, the expression for $l_1$ is obtained using the  normal form formulas in \cite{Kuznetsov2023}.
The unit complex eigenvectors read as
\begin{equation}\label{pqu}
{\mathbf q} = \left( \frac{i}{2\omega},\frac{\omega-i}{2\omega}\right)^T, \qquad 
{\mathbf p} = \left(1+i\omega, 1\right)^T.
\end{equation}
In addition, one  also gets   
\begin{equation*}
\begin{split}
g_{20}&=\:\frac{(\sqrt{p-1}-i)^2(-s+3\sqrt{s^2+8(p-1)})}{8(p-1)},\\
g_{11}&=\:\frac{p(-s+3\sqrt{s^2+8(p-1)})}{8(p-1)},\qquad 
g_{21}=\: -\frac{3(\sqrt{p-1}-i)}{4\sqrt{(p-1)^3}}.
\end{split}
\end{equation*}
Then
 \begin{equation*}\label{l1P}
l_1 =\dfrac{p\:\big(24(p-1)+5s^2-3s\sqrt{s^2+8(p-1)}\: \big)}{32(p-1)^{5/2}} > 0,
\end{equation*}
for all $s\geq 0$ and $p>1$.  
Consequently,  a subcritical bifurcation Hopf occurs at $r= \frac{1}{8} \left(-4+12p-s^2-s\sqrt{s^2+8(p-1)}\right)$, resulting in an unstable limit cycle and a stable equilibrium point at $P_1$.
\end{proof}

\begin{proposition}\text{}
  The equilibrium point  $P_2$ undergoes a subcritical Hopf bifurcation, when 
$s\geq 0$, $p>1$, and $r=\frac{1}{8}\left( -4+12p-s^2+ s\sqrt{-8+8p+s^2}\right)$.
\end{proposition}

The proof is omitted, as it follows directly from arguments analogous to those used in the analysis of the Hopf bifurcation at $P_1$. The nature of the bifurcation is determined by computing the first Lyapunov coefficient, given by
 \begin{equation*}\label{l1P2}
l_1 =\dfrac{p\:\big(24(p-1)+5s^2-3s\sqrt{s^2+8(p-1)}\: \big)}{32(p-1)^{5/2}} > 0,
\end{equation*}
which gives that a subcritical Hopf bifurcation occurs at the $P_2$ for 
$p>1$.

 We remark that the Hopf bifurcations at both $P_1$ and $P_2$ consistently remain subcritical across the parameter ranges considered. This implies that the resulting limit cycles are unstable, leading to transient oscillatory behavior rather than sustained periodic dynamics near these equilibria.  

In parameter regimes where stable and unstable limit cycles coexist, this structure admits a clear physical interpretation. The coexistence reflects a regime of \textit{bistability} between different long-term climate states. The stable cycle represents a self-sustained glacial--interglacial oscillation, while the unstable one acts as a separatrix delimiting its basin of attraction. Consequently, small changes in atmospheric CO$_2$ or ice volume may shift the system between a steady climate state and a regime of recurring glacial cycles. This bistable structure provides a natural mechanism for abrupt transitions between climate states, as small parameter variations may induce rapid shifts between steady regimes and sustained oscillatory behavior.

Figure~\ref{centro_org1} illustrates the bifurcation structure of system~\eqref{MSNcero} in the ($p, r$) parameter plane for $s = 1$. The solid blue curve $e_0$ marks the set of parameter values where the equilibrium at the origin $P_0$ undergoes a Hopf bifurcation. As parameters vary along $e_0$, the criticality of the bifurcation changes from subcritical to supercritical at a point known as the generalized Hopf bifurcation (GH), also known as Bautin bifurcation \cite{Kuznetsov2023}. 
The stability regions of the equilibria $P_0$, $P_1$, and $P_2$  are described as follows: the origin $P_0$ is linearly stable within the region enclosed by the diagonal $r = p$, the curve $e_0$, and the $p$-axis. The equilibrium $P_1$ is stable in the region bounded by the diagonal, the green curve $e_2$, and the $r$-axis, losing stability through a subcritical Hopf bifurcation along $e_2$. Likewise, the equilibrium $P_2$ remains stable within the region enclosed by the red curve $e_1$ and the black dashed curve. The purple dashed curve corresponds to a saddle-node bifurcation, where the equilibria $P_1$ and $P_2$ coalesce, while the black dashed curve represents parameter values for which the equilibrium $P_2$ collides with the origin $P_0$.
The brown curve corresponds to the limit point of cycles (LPC) bifurcation, 
that is, the fold of periodic orbits emerging from the generalized Hopf point where 
$l_1=0$. This curve marks the boundary between parameter regions with one 
and two coexisting limit cycles, completing the classical generalized Hopf bifurcation 
scenario. In this sense, the LPC is the global continuation of the generalized 
Hopf (GH) point in the $(p,r)$-plane.

\begin{figure}[h!]
\begin{center} 
\includegraphics[scale=0.75]{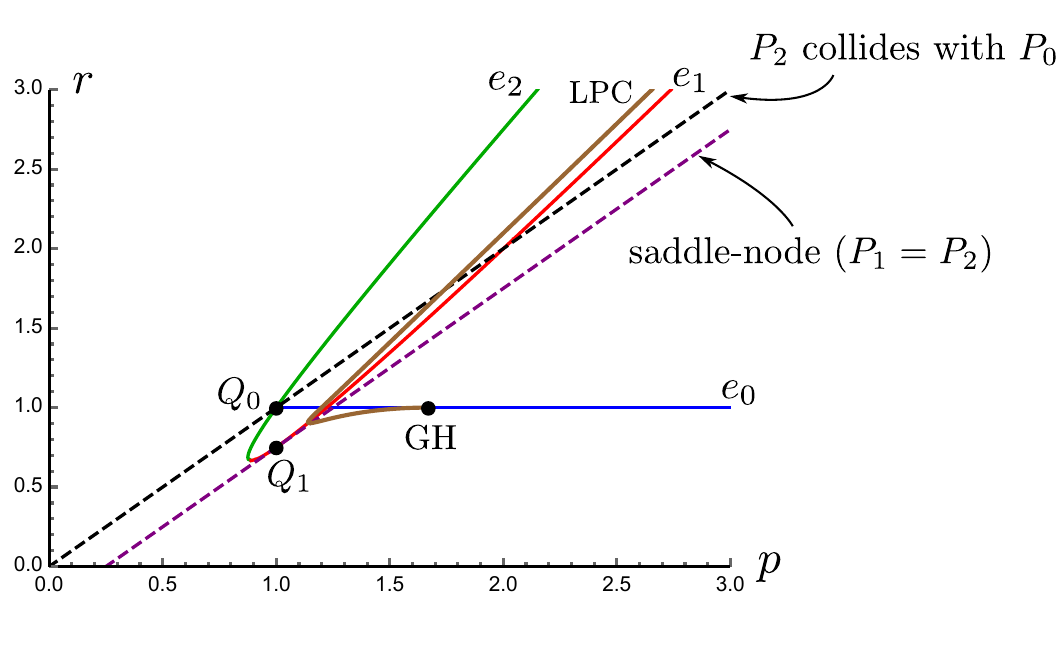} 
\end{center}
\caption{\scriptsize Bifurcation diagram of system \eqref{MSNcero} in the $(p,r)$-plane for $s=1$, showing the stability regions of the equilibria $P_0$, $P_1$, and $P_2$, together with the main local and global bifurcations.
The curves $e_0$ (blue), $e_1$ (red), and $e_2$ (green) correspond to Hopf bifurcation loci at the equilibria $P_0$, $P_2$, and $P_1$, respectively. The points $Q_0=(1,1)$ and $Q_1$ denote Takens--Bogdanov bifurcation points (organizing centers). The diagonal $r=p$, which passes through $Q_0$, separates regions with different numbers of equilibria. The purple dashed line represents the saddle-node bifurcation curve along which the equilibria $P_1$ and $P_2$ coalesce. The black dashed line corresponds to parameter values for which the equilibrium $P_2$ collides with the origin, while $P_1$ persists at $(s,-s)$.
The equilibrium $P_0$ is linearly stable in the region bounded by the diagonal $r=p$, the curve $e_0$, and the $p$-axis. The equilibria $P_1$ and $P_2$ are stable in the regions delimited by $e_2$ and $e_1$, respectively, together with the corresponding boundaries described above. In all cases, stability is lost through Hopf bifurcations along these curves.
The point GH on $e_0$ denotes a generalized Hopf (Bautin) bifurcation, where the associated Hopf bifurcation changes its criticality from subcritical to supercritical. The brown curve (LPC) corresponds to the limit point of cycles bifurcation, representing the fold of periodic orbits and delimiting the parameter region where two limit cycles (one stable and one unstable) coexist. The intersection of the curve $e_1$, the saddle-node curve, and the LPC curve corresponds to parameter values where multiple bifurcation conditions are simultaneously satisfied.}
\label{centro_org1}
\end{figure}

In order to numerically observe the limit cycles emanating from the generalized Hopf point, we performed a numerical exploration along the Hopf bifurcation curve $e_0$, that is, 
for $r = 1$ and fixed $s$, while varying $p$. This analysis enabled us to detect and characterize the emergence of limit cycles in the system. The resulting phase portraits, shown in Figure~\ref{fig-zoo1}, illustrate the qualitative dynamics near the bifurcation curve.
\begin{figure}[hbt]
\begin{center} 
\hspace{0.2cm}
\subfigure[\; $p = 1.4$, $r = 1$, $s=1$]{
\includegraphics[scale=0.55]{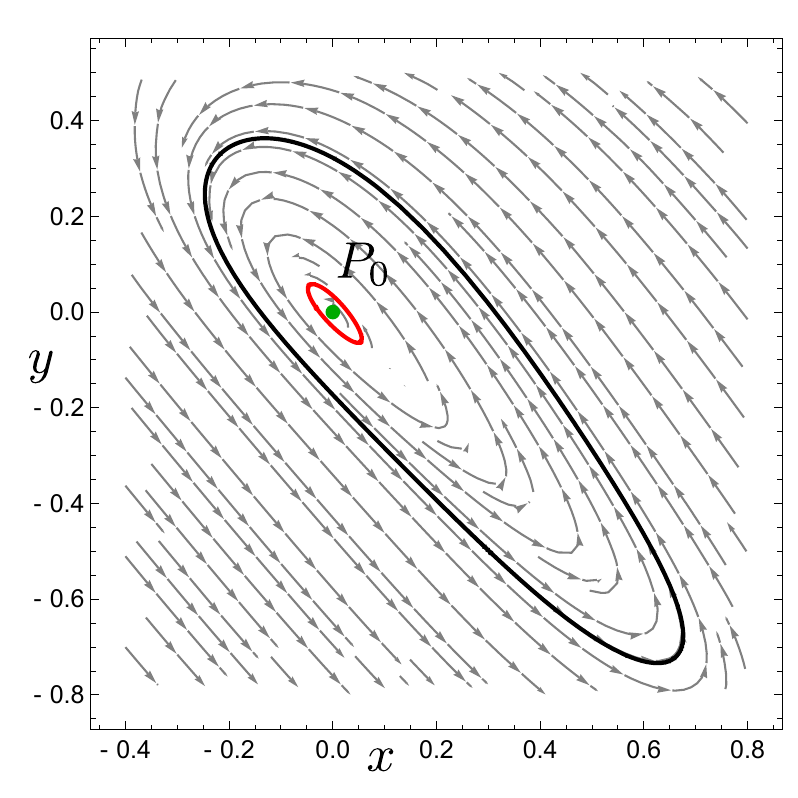}}
\hspace{0.2cm}
\subfigure[\; $p = 1.85$, $r = 1.1$, $s=1$]{
\includegraphics[scale=0.55]{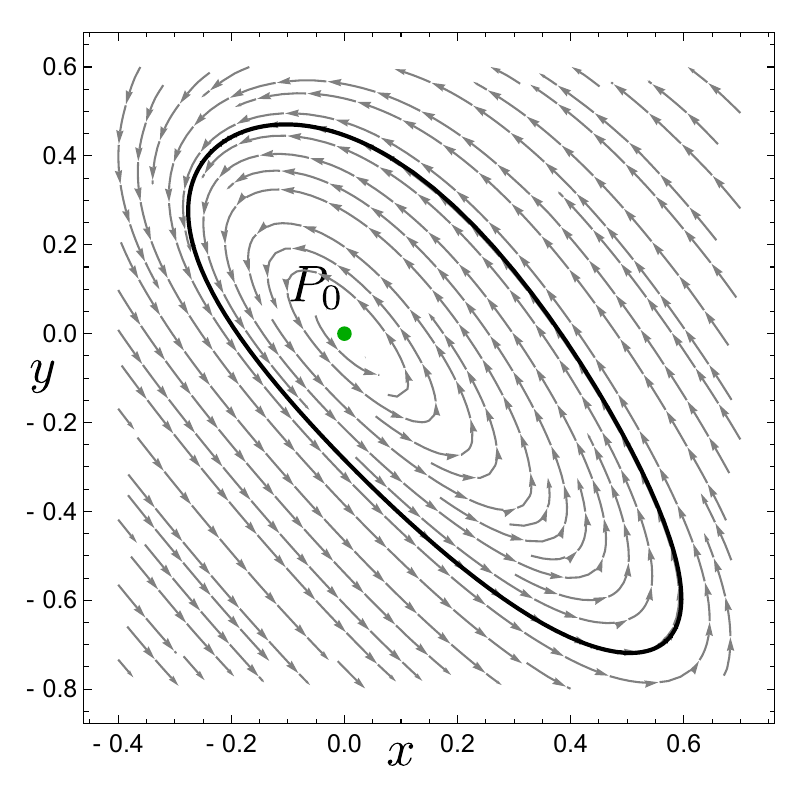}}
\hspace{0.2cm}
\end{center}
\caption{Phase portraits of system \eqref{MSNcero} in the $(x,y)$-plane, illustrating the emergence of different types of limit cycles along the curve $e_0$ in Fig.~\ref{centro_org1}. The green dot denotes the origin $P_0$. (a) two limit cycles, consisting of a stable one (black) and an unstable one (red) surrounding the origin; (b) a single stable limit cycle, with the unstable one having vanished.}\label{fig-zoo1}
\end{figure}

Before proceeding further, we recall that, in general, the existence of a homoclinic orbit implies the presence of oscillatory dynamics near the associated saddle point, due to the interaction between its stable and unstable manifolds \cite{Guckenheimer}. To illustrate this behavior in system \eqref{MSNcero}, we fixed representative parameter values and computed homoclinic trajectories using {\em Mathematica}. The computation involves matching the unstable manifold of one equilibrium point with the stable manifold of another to reconstruct the full homoclinic loop. The resulting homoclinic orbits are displayed in Figures~\ref{homp0}, \ref{homp2}, and \ref{homp1p2}.

In particular, point $P_1$ (highlighted in red) maintains consistent stability in different dynamical scenarios. As shown in Figures~\ref{homp0} and \ref{homp2}, it appears either as a stable focus or as a stable node, underscoring its robustness despite variations in the global phase portrait.
\begin{figure}[h!]
\begin{center} 
\includegraphics[scale=0.5]{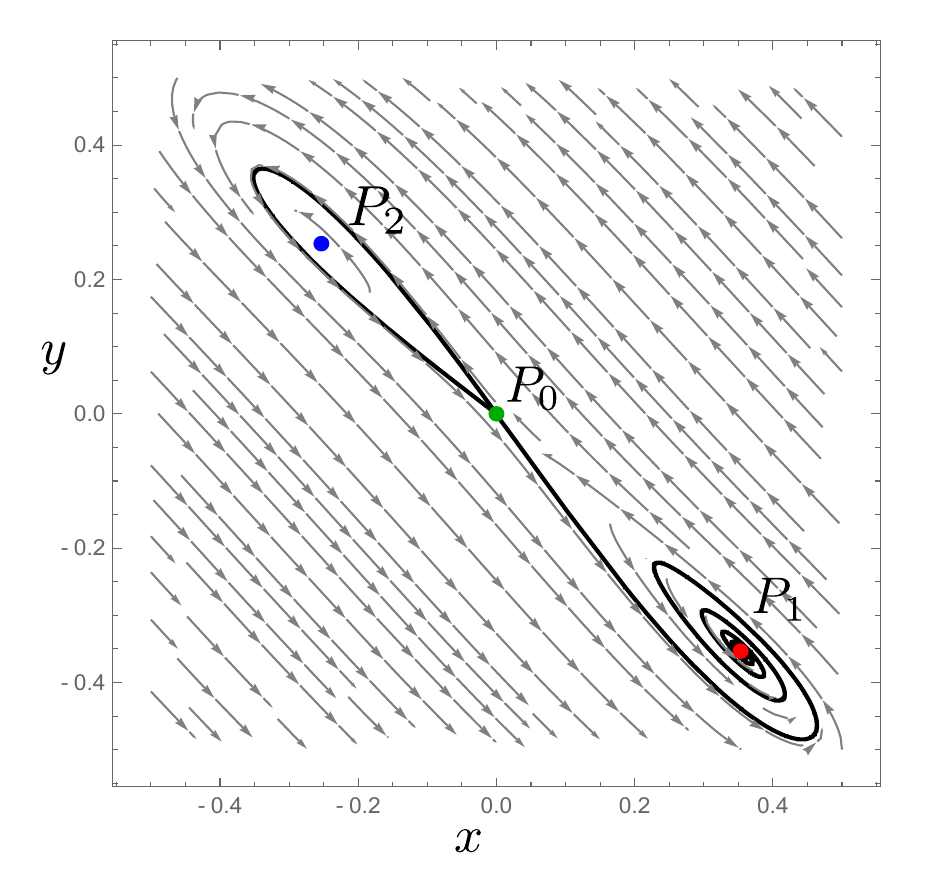} 
\end{center}
\caption{A homoclinic orbit of the system \eqref{MSNcero} emerging from the origin $P_0$ (green) and enclosing the equilibrium point $P_2$ (blue), while  $P_1$ (red) is a stable focus, for $p=1.1104755$, $r=1.2$ and $s=0.1$.}
\label{homp0}
\end{figure}

\begin{figure}
\begin{center}
\subfigure[\; Homoclinic orbit]{
\includegraphics[scale=0.46]{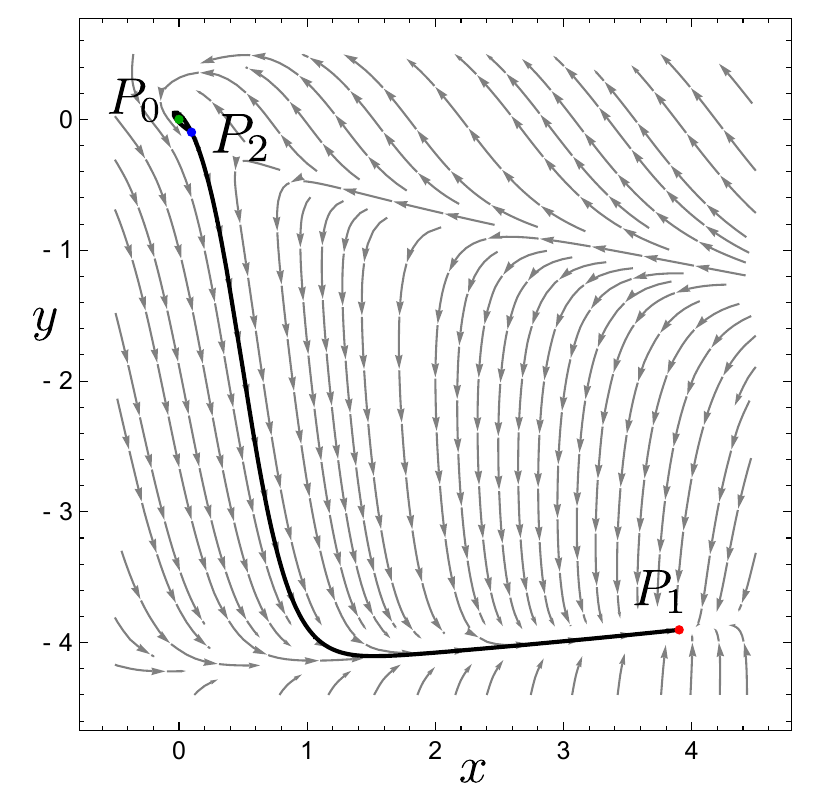}}\hspace{0.2cm}
\hspace{0.2cm}
\subfigure[\; Magnified view of Figure (a)]{
\includegraphics[scale=0.48]{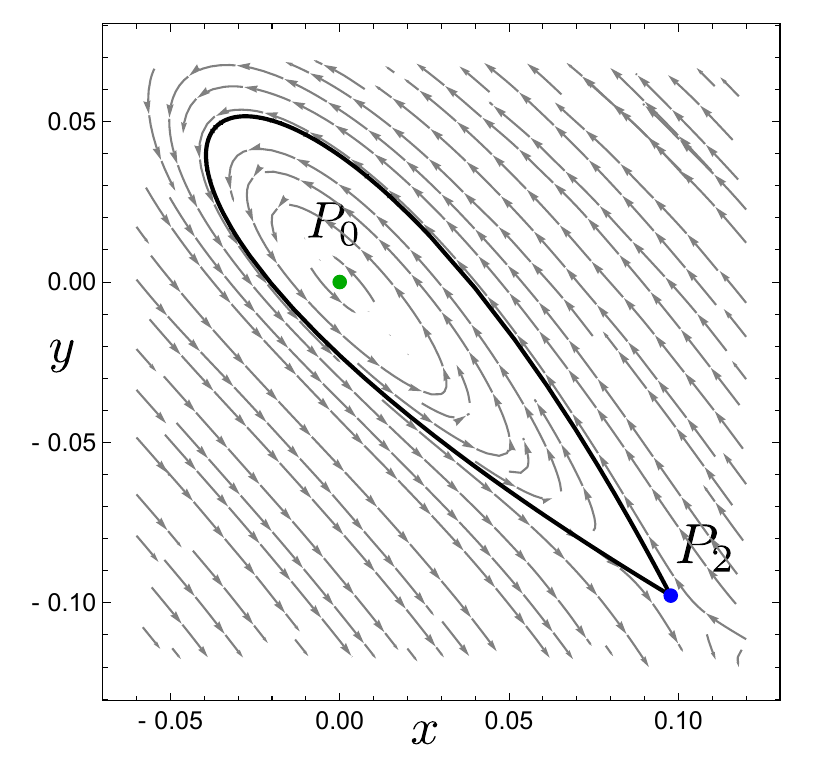}}
\end{center}
\caption{Homoclinic orbit of the system \eqref{MSNcero} emerging from the $P_2$ (blue) and enclosing the origin  $P_0$ (green), for $p=1.281409$, $r=0.9$ and $s=4$. The point 
$P_1$ (red) is a stable node.}
\label{homp2}
\end{figure}

\begin{figure}[h!]
\begin{center} 
\includegraphics[scale=0.6]{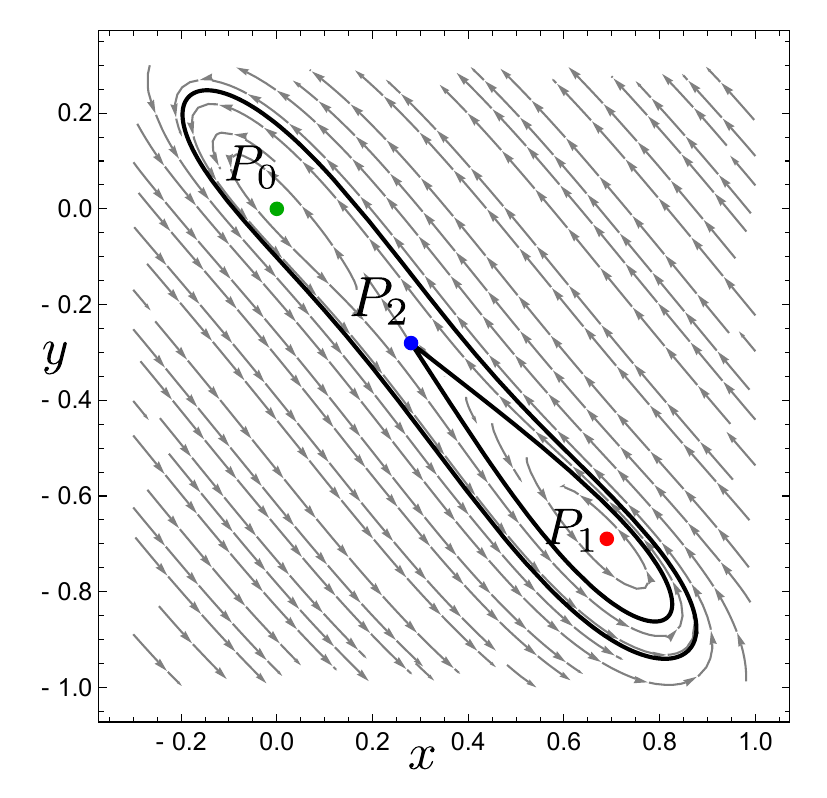} 
\end{center}
\caption{Homoclinic orbit of the system \eqref{MSNcero} emerging from the origin $P_2$ (blue) and enclosing the equilibrium point $P_1$ (red), while  the origin $P_0$ (green) is a unstable focus, for $p=1.20339372$, $r=1.01$ and $s=0.97$.}
\label{homp1p2}
\end{figure}
\newpage

\section{Bifurcation unfolding on the critical manifold}\label{sec_hamilton}
To analyze the onset and nature of the oscillatory regimes, we focus on the local unfolding of
the Hopf and generalized Hopf bifurcations that organize the dynamics on the critical manifold. As we detail below, a
change of variables and an appropriate scaling are introduced to express the system in a
canonical slow-fast form suitable for normal form analysis. This transformation is not merely
algebraic: it balances the characteristic time and amplitude scales of the ice, CO$_2$, and
ocean components, allowing the interaction among these subsystems to be described in
nondimensional units of comparable magnitude. Physically, the rescaled variables represent
deviations from the mean climate state rather than absolute quantities, and the scaling makes
the feedback structure between the slow and fast variables explicit. This approach provides a
clear physical interpretation while preserving the analytical tractability required to derive
the canonical unfolding of the system. Based on this local formulation, the following analysis
explores how these local bifurcations interact and organize the global dynamics of the model.

Before proceeding, we briefly comment on the role of the parameter $s$. The case $s=0$ is mathematically convenient, as it introduces a symmetry that simplifies the bifurcation analysis. However, this situation is not physically realistic, since nonlinear feedback mechanisms in the carbon cycle are expected to break this symmetry.

When $s \neq 0$, this symmetry is lost, leading to an asymmetric phase space structure and corresponding changes in the homoclinic orbits. This affects the robustness of abrupt climate transitions, as the system becomes more sensitive to parameter variations. In particular, asymmetric configurations can produce transitions with different amplitudes and durations, which provides a more realistic description of glacial--interglacial variability.

In what follows, we explore the global bifurcation structure of the Maasch-Saltzman model near two organizing centers in the $(p, r)$-parameter plane, where codimension-two Takens-Bogdanov bifurcations occur. In particular, we unfold the Takens-Bogdanov bifurcation at the origin $P_0$, which serves as a central point organizing local and global dynamics such as Hopf and saddle-node bifurcations, as well as homoclinic orbits. Unfolding it enables us to understand how periodic orbits emerge or disappear, and how their stability changes as system parameters vary. To analyze the persistence of the homoclinic orbits, we employ a Hamiltonian approximation together with Melnikov theory.

The linearization of system \eqref{MSNcero} at the origin $(0,0)$ for $p = r = 1$, and at the equilibria $P_1$ and $P_2$ for $p = 1$ and $r = 1 - \tfrac{1}{4}s^2$, reveals a non-semisimple double-zero eigenvalue: a zero eigenvalue with algebraic multiplicity two and geometric multiplicity one. This singularity corresponds to a codimension-two Takens-Bogdanov bifurcation. For each fixed value of the parameter $s$, these bifurcations define two organizing centers in the $(p,r)$-parameter plane, located at:
$$
Q_0 = (1,1), \qquad Q_1 = \left(1,1 - \tfrac{1}{4}s^2\right).
$$
These points serve as critical loci around which complex bifurcation structures emerge. As illustrated in Figure~\ref{centro_org1}, which corresponds to the case $s = 1$, the codimension-two bifurcation unfolds in the $(p,r)$-plane, highlighting the organizing centers $Q_0$ and $Q_1$ along with the branches of both local and global bifurcations of system \eqref{MSNcero}.

Since the Takens-Bogdanov bifurcation is the simplest (i.e., lowest codimension) bifurcation capable of generating homoclinic orbits, we proceed with the unfolding of the organizing center $Q_0$ in \eqref{MSNcero}. A similar analysis applies to the point $Q_1$, as it exhibits analogous local behavior.

Our analysis focuses on homoclinic orbits arising from an appropriate Hamiltonian system
and use Melnikov theory  to determine the parameter sets
for which these homoclinics persist under small perturbations (see \cite{Guckenheimer} and \cite{Perko}).
 This approach provides a deeper understanding of the global bifurcation structure near the organizing centers and reveals the mechanisms responsible for the emergence of oscillatory dynamics in the model.

To express the system in a canonical form suitable for the application of Melnikov theory, we introduce the change of variables $(x,-x-y)\to(x,y)$.  This transformation brings the system~\eqref{MSNcero} into a form in which the leading-order dynamics can be written as a Hamiltonian system, allowing the persistence of homoclinic orbits under perturbations to be analyzed.
From a physical perspective, this transformation reflects the approximate balance between the ice and deep-ocean components: when the ice mass anomaly $x$ increases, the ocean temperature anomaly $-x - y$ decreases, representing a compensating response in the oceanic heat reservoir. At the same time, it is introduced to rewrite the system in a canonical form suitable for Melnikov analysis, where the leading-order dynamics becomes Hamiltonian.

Rewriting the system in terms of $(x, y)$ isolates this negative feedback and yields a quasi-Hamiltonian structure that allows an explicit analysis of the homoclinic orbits and their persistence under perturbation.
Under this transformation, system~\eqref{MSNcero} becomes:
\begin{equation}
\label{homo}
\begin{split}
\dot{x}&=\: y,\\
\dot{y}&=\: (r-p)x+(r-1)y+s(x+y)^2-(x+y)^3.
\end{split}
\end{equation}

We proceed by introducing a time rescaling $\tilde{t} = \eta t$ and 
performing a change of variables $x(t) = \eta u(\tilde{t})$ and 
$y(t) = \eta^2 v(\tilde{t})$. To avoid introducing yet another time notation, we emphasize that from this point on, the overdot denotes 
differentiation with respect to $\tilde{t}$ rather than $t$. In parallel, 
we rescale the parameters as
\begin{equation}\label{param}
\delta = \frac{s}{\eta}, \quad \mu = \frac{r - p}{\eta^2},\quad \text{and}\quad  \lambda = \frac{r - 1}{\eta^2}.
\end{equation}
Under these transformations, system \eqref{homo} becomes equivalent to:
\begin{equation}
\label{hamhomo}
\begin{split}
\dot{u}&=\: v,\\
\dot{v}&=\: \mu u+\delta u^2-u^3+\eta(\lambda v+2\delta uv-3u^2v)+\eta^2(\delta v^2-3uv^2)-\eta^3 v^3.
\end{split}
\end{equation}

We emphasize that, although system \eqref{hamhomo} is not strictly Hamiltonian due to the presence of dissipative terms, it admits a Hamiltonian structure when $\eta=0$.
The dynamics are governed by Hamilton’s equations
\begin{equation}
\label{Ham}
\begin{split}
\dot{u}&=\: v,\\
\dot{v}&=\: \mu u +\delta u^2-u^3,
\end{split}
\end{equation}
with associated Hamiltonian function
\begin{equation}
\label{Hamm}
H(u,v)=\dfrac{1}{2}v^2-\dfrac{\mu}{2}u^2-\dfrac{\delta}{3}u^3+\dfrac{1}{4}u^4.
\end{equation}
We regard system \eqref{Ham} as the unperturbed counterpart of \eqref{hamhomo}. It admits three equilibrium points:  $p_0= (0,0)$, 
$$p_1 = \left(\tfrac{1}{2}(\delta + \sqrt{\delta^2 + 4 \mu}), 0\right) \quad  \text{and} \quad p_2 = \left(\tfrac{1}{2}(\delta - \sqrt{\delta^2 + 4\mu}), 0\right),$$ for $\delta^2 + 4\mu > 0$.  
Their linear stability is determined by evaluating the Jacobian matrix at each point and analyzing the corresponding eigenvalues. The results are summarized below.
\begin{proposition}Consider system \eqref{Ham}. The nature of its equilibrium points is characterized as follows:
\begin{enumerate}
\item[{1.} ] The origin is  a saddle point if $\mu >0$, and a centre if $\mu < 0$.
\item [{2.} ] The equilibrium point $p_1$ is a center if $\delta>0$ and $\mu+\frac{\delta^2}{4}>0$.
\item[{3.} ]  The equilibrium point $p_2$ is a saddle point if either $\delta >0$ and $-\frac{\delta^2}{4}<\mu<0$. It is a center if $\delta >0$ and $\mu>0$.
\end{enumerate}
\end{proposition}

We note that all equilibria exhibit degenerate behavior, characterized by a double-zero eigenvalue, when $\mu = 0$ (which corresponds to the condition $r = s$). Figure~\ref{fig-phase} illustrates the evolution of the phase portraits of system \eqref{Ham} as the parameters $\mu$ and $\delta$ vary, highlighting the qualitative bifurcation phenomena that emerge. In particular, we observe the presence of double-loop homoclinic orbits.

The emergence of homoclinic orbits in the Hamiltonian system is a direct consequence of the saddle nature of the origin when $\mu > 0$, along with the existence of appropriately positioned nearby centers, which facilitate the reconnection of stable and unstable manifolds. These global structures are clearly visible in the phase portraits and motivate a detailed analysis of their persistence under small non-Hamiltonian perturbations. This will be addressed in the following sections using Melnikov theory.
\begin{figure}[h!]
\begin{center} 
\subfigure[\; $\mu = 2$, $\delta = 4$]{
\includegraphics[scale=0.5]{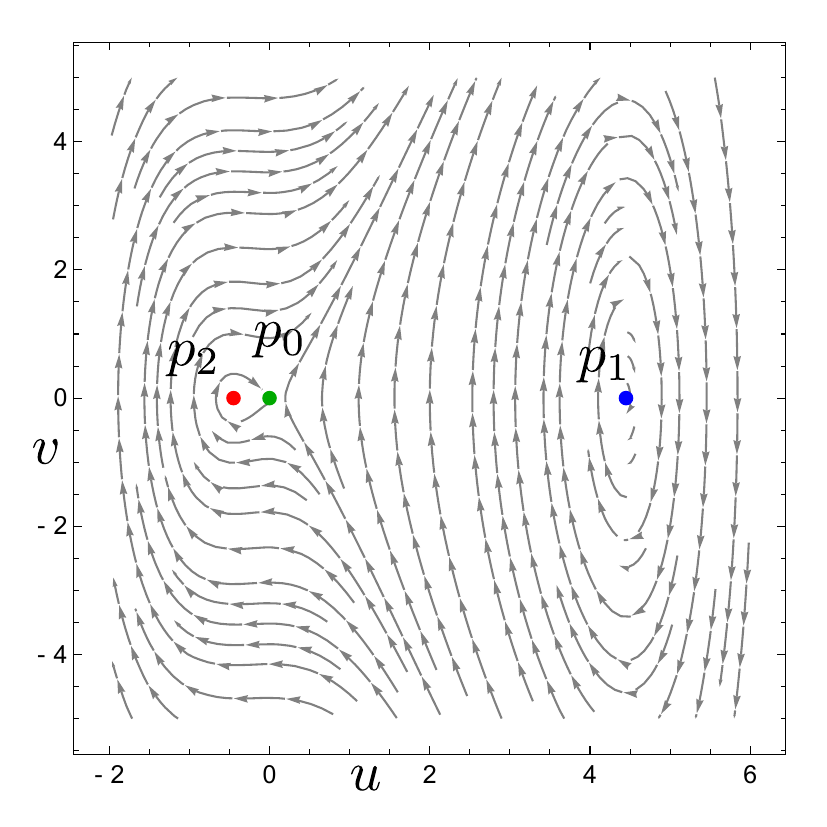}}\hspace{0.3cm}
\subfigure[\; $\mu = 0$, $\delta = 4$]{
\includegraphics[scale=0.5]{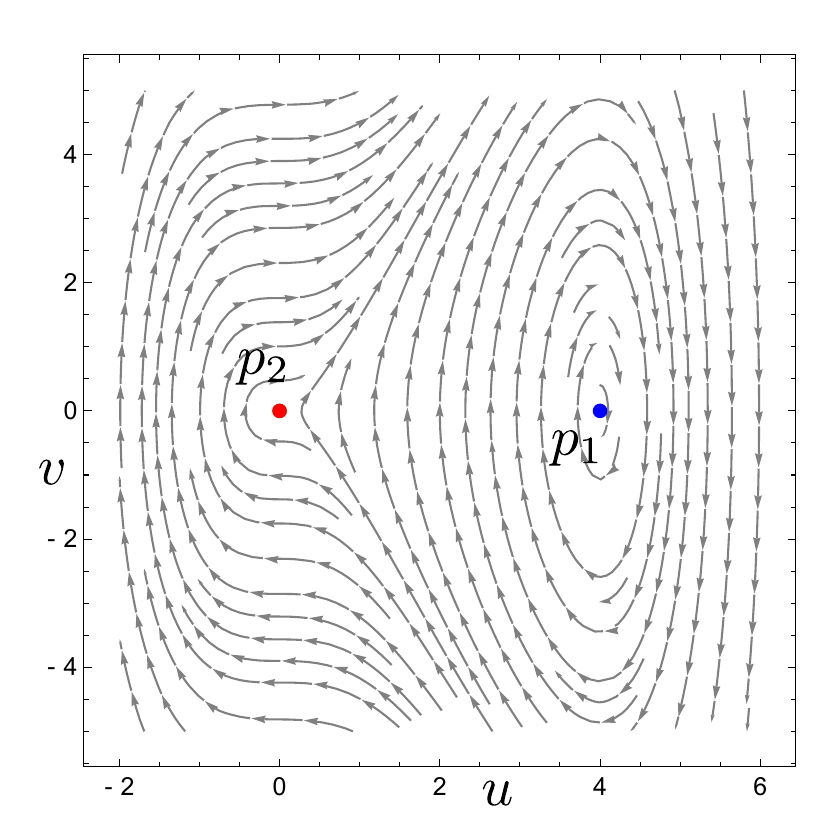}}
\end{center}
\caption{Phase portraits of the Hamiltonian system \eqref{Ham} for representative values of the parameters $\mu$ and $\delta$, illustrating the corresponding bifurcation scenarios in the $(u,v)$-plane. The equilibrium points $p_0$, $p_1$, and $p_2$ are shown in green, blue, and red, respectively.}\label{fig-phase}
\end{figure}

We proceed to investigate the effect of small perturbations on the persistence of homoclinic orbits to the origin in the Saltzman--Maasch climate model. These orbits, which connect a saddle equilibrium to itself, play a central role in capturing the nonlinear mechanisms underlying glacial--interglacial variability. In particular, they generate large-amplitude variations in the system’s state variables, modeling transitions between quasi-stable climate states. From a physical perspective, such transitions can be interpreted as abrupt shifts between climate states, driven by nonlinear feedback mechanisms in the system.

To assess the robustness of these structures, we consider the homoclinic orbits of the unperturbed Hamiltonian system \eqref{Hamm} and study their behavior under small non-Hamiltonian perturbations. Our approach is based on Melnikov theory (see Chapter 4 of \cite{Guckenheimer}), which provides a rigorous criterion for detecting transverse intersections between stable and unstable manifolds. The presence or absence of such intersections determines whether the homoclinic connection persists or breaks, thereby revealing the sensitivity of these global structures to perturbations.
\bigskip 

The unperturbed system \eqref{Hamm} admits two homoclinic orbits to the hyperbolic saddle at the origin, both lying on the level set $H(x, y) = 0$.  These homoclinic trajectories can be explicitly parametrized in time, and their analytic  expressions are given by
\begin{equation}\label{gam1}
\Gamma^+_0:  (u^+(\tilde{t}),v^+(\tilde{t})) = \left( -\frac{3\sqrt{2}\mu \alpha}{\sqrt{2}\delta\alpha-\sqrt{\mu}\cosh{(\sqrt{\mu}\:\tilde{t})}},-\frac{3\sqrt{2}\mu^2\alpha\sinh{(\sqrt{\mu}\:\tilde{t})}}{\left(\sqrt{2}\delta\alpha-\sqrt{\mu}\cosh{(\sqrt{\mu}\:\tilde{t})}\right)^2}\right),
\end{equation}
and
\begin{equation}\label{gam2}
\Gamma_0^-:  (u^-(\tilde{t}),v^-(\tilde{t})) = \left( -\frac{3\sqrt{2}\mu\alpha}{\sqrt{2}\delta\alpha+\sqrt{\mu}\cosh{(\sqrt{\mu}\:\tilde{t})}},\frac{3\sqrt{2}\mu^2\alpha \sinh{(\sqrt{\mu}\:\tilde{t})}}{\left(\sqrt{2}\delta\alpha+\sqrt{\mu}\cosh{(\sqrt{\mu}\:\tilde{t})}\right)^2} \right), 
\end{equation}
where $\alpha = \sqrt{\frac{\mu}{2\delta^2+9\mu}}$ with $\mu<-\frac{2\delta^2}{9}$ or $\mu>0$ and $\delta\geq 0$.

The phase portrait of \eqref{Ham} for $\mu =  2$ and $\delta= 4$, including the homoclinic orbits, is shown in Figure \ref{fig-phase2}. Henceforth, we shall refer to $\Gamma_0^-$ as the left loop and to the $\Gamma_0^+$ right loop.
\begin{figure}[hbt]
\centering
\subfigure[\; $\Gamma_0^-$ with $\mu = 2$ and  $\delta = 4$.]{
\includegraphics[scale=0.55]{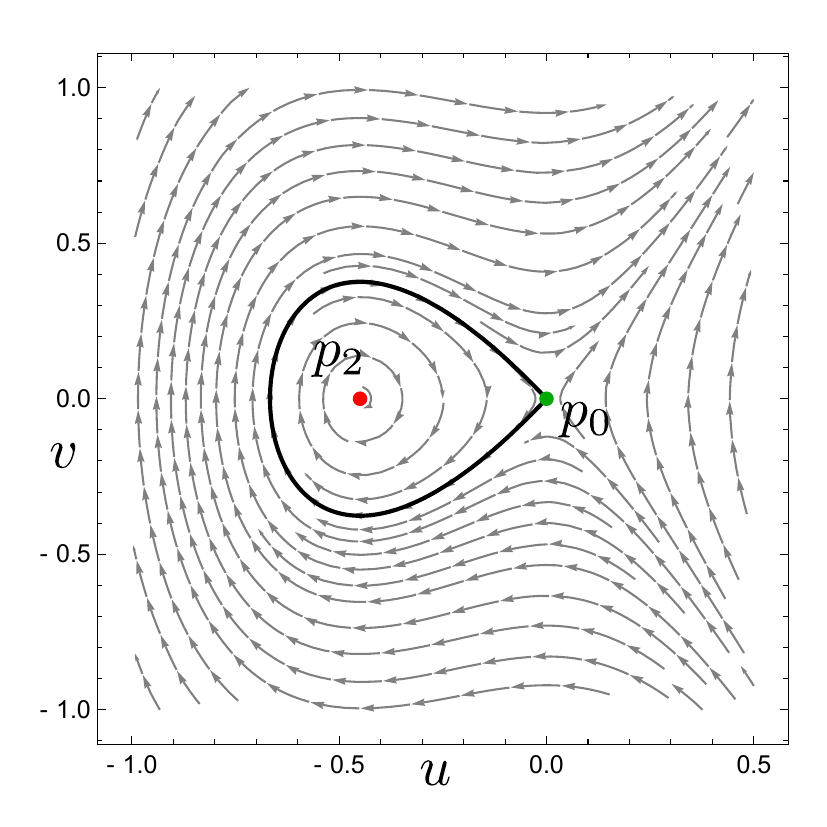}\hspace{0.8cm}}
\subfigure[\; $\Gamma_0^+$ with $\mu = 2$ and  $\delta = 4$.]{
\includegraphics[scale=0.55]{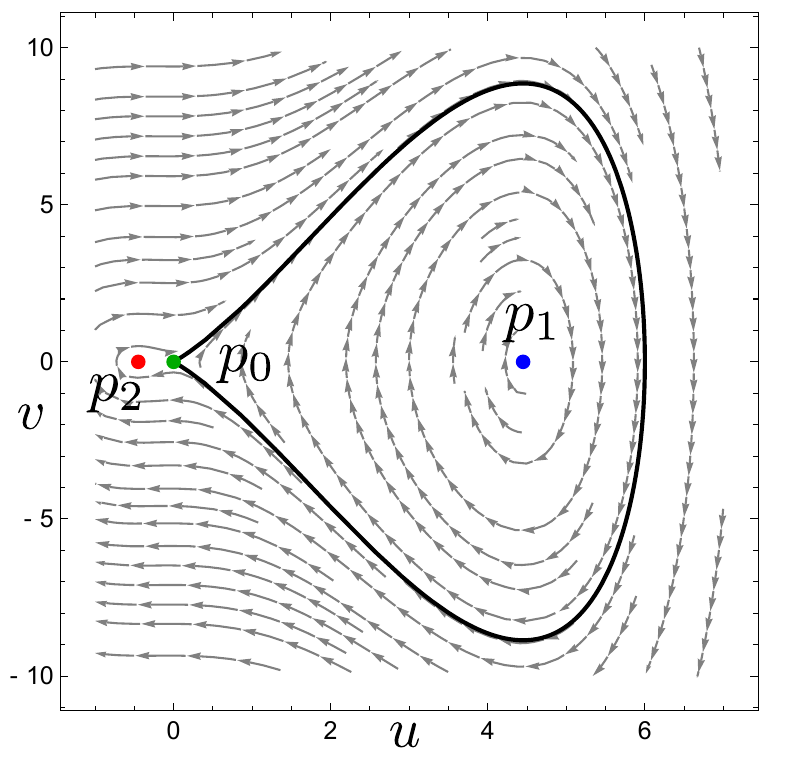}}
\caption{Phase portrait of the Hamiltonian system \eqref{Ham} in the $(u,v)$-plane.
It displays a homoclinic trajectory connected to a saddle point at the origin, enclosing the points $p_1$ and $p_2$, which are centers.}\label{fig-phase2}
\end{figure}

\begin{figure}[hbt]
\begin{center} 
\includegraphics[scale=0.8]{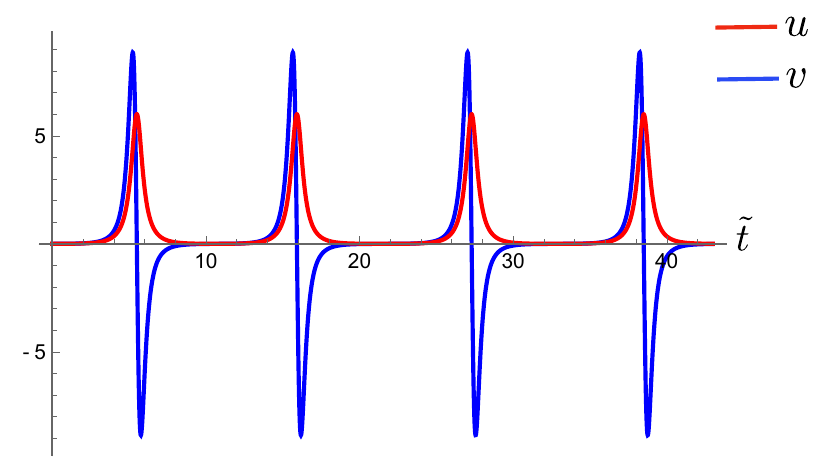} 
\end{center}
\caption{The time series corresponding to the homoclinic orbit shown in  \ref{fig-phase2}(a) display $u = u(\tilde{t})$ in red and $v = v(\tilde{t})$ in blue, both exhibiting oscillatory behavior.}
\label{Ser-hom}
\end{figure}

Once the homoclinic orbits have been parametrized, we proceed to compute the associated Melnikov function, which yields a first-order approximation to the splitting distance between the stable and unstable manifolds of the perturbed system. This analysis is crucial for determining the persistence of the homoclinic connection under small perturbations: the existence of a simple zero of the Melnikov function implies a transverse intersection of the perturbed invariant manifolds, thereby guaranteeing the survival of a homoclinic orbit in the non-autonomous (or perturbed) setting \eqref{hamhomo} for $0< \eta \ll 1$; see, e.g., \cite{Guckenheimer}.

To investigate whether the homoclinic orbits $\Gamma_0^+$  and $\Gamma_0^-$ persist under perturbation, we introduce a distance function $D(\delta, \mu, \lambda)$ that measures the separation between the stable and unstable manifolds in the perturbed system. This function is constructed with reference to the level sets of the Hamiltonian $H$. As a first step, we recast system \eqref{hamhomo} into the form of a perturbed Hamiltonian system:
\begin{equation}\label{Hamestr}
\begin{array}{lll}
\begin{pmatrix}
\dot{u}\\
\dot{v}\\
\end{pmatrix} &=&
\begin{pmatrix}
v\\
\mu u+\delta u^2-u^3\\
\end{pmatrix} + \eta
\begin{pmatrix}
0\\
v(\lambda+2\delta u-3u^2)
\end{pmatrix} + \mathcal{O}(\eta^2) \vspace{0.3cm}\\
&=& J\nabla H+\eta X+\mathcal{O}(\eta^2), 
\end{array}
\end{equation}
where 
$J=\begin{pmatrix}
0 & 1\\
-1 & 0\\
\end{pmatrix}$.

The leading-order behavior of the splitting between the perturbed stable and unstable manifolds is captured by an asymptotic expansion in powers of $\eta$, describing the infinitesimal deviation induced by the perturbation:
 $$D^\pm(\delta,\mu,\lambda)=\eta M^\pm(\delta,\mu, \lambda)+\mathcal{O}(\eta^2),$$ 
  where the is Melnikov function   given by
 \begin{equation}\label{mell1}
 M^\pm(\delta,\mu, \lambda)=\int_{-\infty}^{\infty}{\nabla H\cdot X(\Gamma_0^\pm(\tilde{t}))ds}.
 \end{equation}

Using \eqref{Hamm}, \eqref{Hamestr}, and the expressions for $\Gamma_0$ given in \eqref{gam1} and \eqref{gam2}, we obtain the Melnikov function
\[
M^\pm(\delta,\mu,\lambda)
= \lambda I_0^\pm(\delta,\mu) + 2\delta I_1^\pm(\delta,\mu) - 3 I_2^\pm(\delta,\mu),
\]
where the integrals $I_0^\pm$, $I_1^\pm$, and $I_2^\pm$ are defined in Appendix~\ref{ap_melnikov}.

We stress that the presence of simple zeros in the Melnikov function implies a transverse intersection between the stable and unstable manifolds, thereby ensuring the persistence of the homoclinic orbit under perturbation. In this context, if $\eta > 0$ is sufficiently small, the homoclinic orbit $\Gamma_0^\pm$ to the origin $p_0$ persists for parameter values such that the equation $D^\pm(\delta,\mu,\lambda)=0$ admits simple zeros.
The corresponding leading-order persistence condition is therefore given by
\[
\lambda = \lambda_0^\pm(\delta,\mu)
= \dfrac{3I_2^\pm (\delta,\mu)-2\delta I_1^\pm(\delta,\mu)}{I_0^\pm(\delta,\mu)}
+ \mathcal O(\eta),
\]
which provides an explicit relation between the parameters $(\delta,\mu,\lambda)$ ensuring the survival of the homoclinic connection under perturbations.

At leading order, the condition $D^\pm(\delta,\mu,\lambda)=0$ is equivalent to $M^\pm(\delta,\mu,\lambda)=0$, and therefore determines the persistence of the homoclinic orbits.

The leading-order approximation in $\eta$ for the persistence of the homoclinic orbit referred to as the left loop, expressed explicitly in terms of the parameters $\delta$ and $\mu$, is given by
\begin{align}
    \lambda=\lambda^-& =\frac{5 \sqrt{2} \delta  \left(2 \delta ^2+9 \mu \right)^2\text{arctan}(\alpha )+3 \sqrt{\mu } \left(10 \delta ^4+75 \delta ^2 \mu +108 \mu ^2\right)}{15 \left(\sqrt{2} \delta   \left(2 \delta ^2+9 \mu \right)\text{arctan}(\alpha )+3 \sqrt{\mu } \left(\delta ^2+3 \mu \right)\right)},
\end{align}
where $\alpha=\frac{\sqrt{2}\:\delta-\sqrt{2\delta^2+9\mu}}{3\sqrt{\mu}}$.

Figure \ref{persistencia3par} displays the surface defined by $D^-(\delta, \mu, \lambda) = 0$, providing numerical evidence for the existence of homoclinic bifurcation curves associated with the equilibrium point $p_0$. To further illustrate the persistence of the homoclinic orbit in system \eqref{hamhomo}, we selected representative parameter values from this surface: $\delta = 4$, $\mu = 2$, 
$\lambda = 3.6037$, and 
$\eta = 0.3$, and numerically reconstructed the corresponding homoclinic trajectory. The resulting phase portrait, shown in Figure \ref{retratos}, depicts the orbit in the
$(u,v)$-plane.
To verify that this homoclinic orbit also persists in the original system \eqref{MSNcero}, we used the equivalent parameter values $p = 1.4431$, $r = 1.32431$, and $s = 1.2$. The resulting trajectory confirms the existence of the homoclinic orbit in the original variables and is shown in Figure \ref{homoxy}.
\begin{figure}[h!]
\centering
{\includegraphics[scale=0.6]{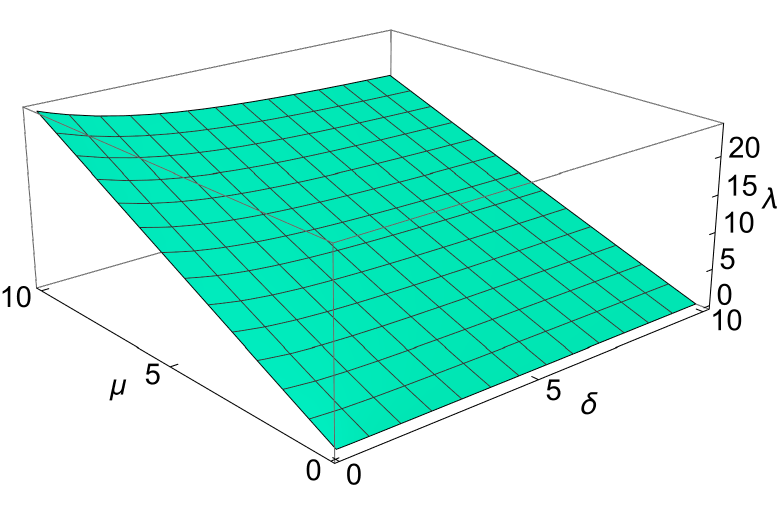}}
\caption{Plot of the surface $D^-(\delta, \mu, \lambda) = 0$, which corresponds to the set of parameter values for which the stable and unstable manifolds intersect after perturbation of the homoclinic orbit $\Gamma_0^-$, indicating its persistence under the perturbation.}\label{persistencia3par}
\end{figure}

\begin{figure}[h!]
\centering
\subfigure[\; $\eta=0$]{
\includegraphics[scale=0.5]{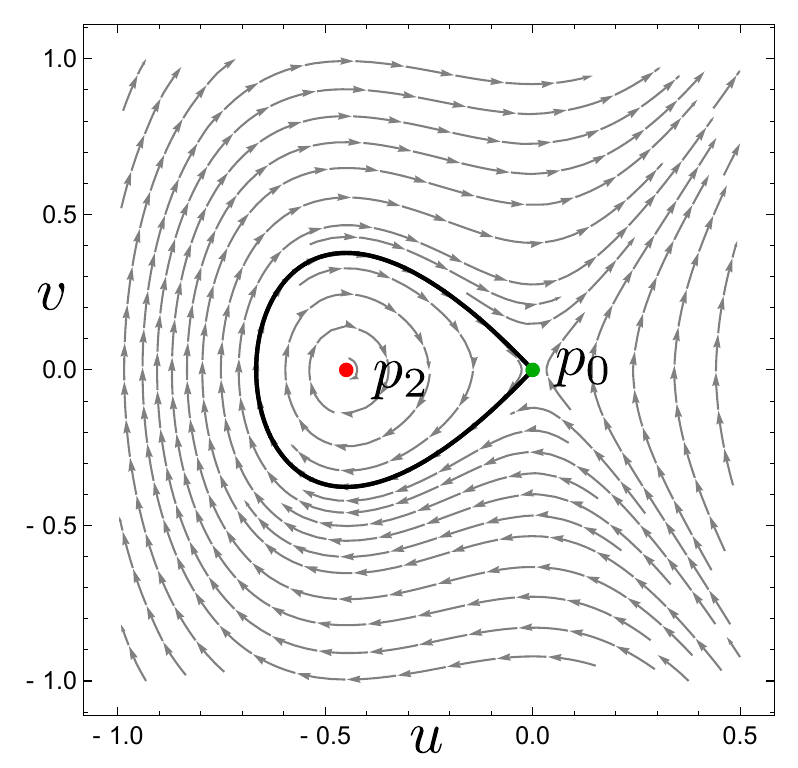}\hspace{0.8cm}}
\subfigure[ \; $\eta=0.3$]{
\includegraphics[scale=0.5]{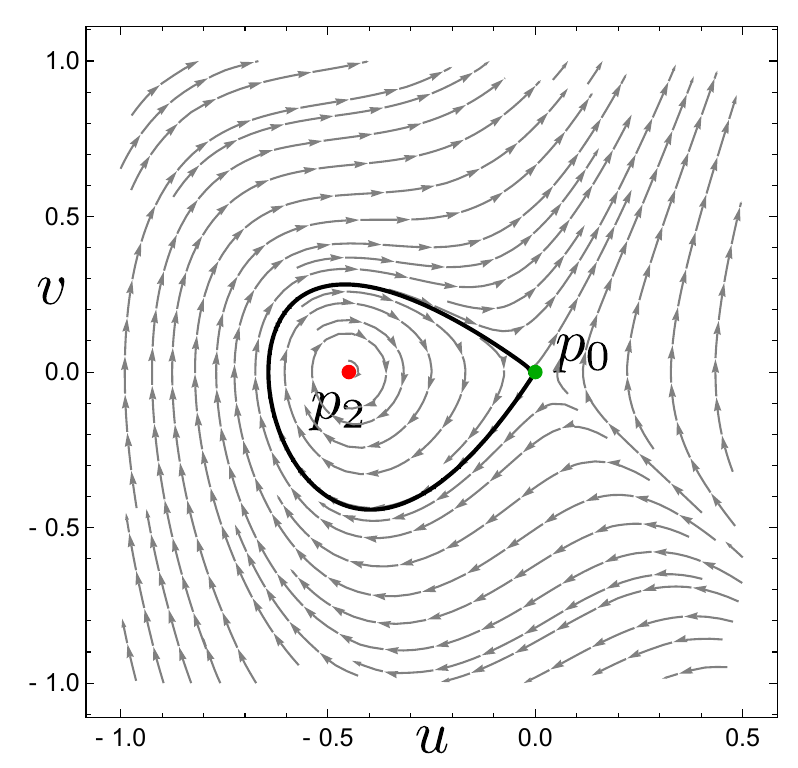}}
\caption{The homoclinic orbit $\Gamma_0^-$ of the Hamiltonian system~\eqref{Ham} persists in the perturbed system~\eqref{hamhomo}, illustrated in the $(u,v)$-plane for parameters $\delta = 4$, $\mu = 2$, and $\lambda = 3.6037$. 
(a) For $(\eta = 0)$ (unperturbed case), given by 
$\Gamma_0^- = \left( -\frac{6}{5 \cosh(\sqrt{2} t)+4}, \frac{30\sqrt{2} \sinh(\sqrt{2} \tilde{t})}{(5 \cosh(\sqrt{2} t)+4)^2} \right)$, where the equilibrium point $p_2$ enclosed by the loop is a center; (b) Numerical solution of the perturbed system \eqref{hamhomo} for $\eta = 0.3$, where $p_2$ becomes asymptotically stable.}\label{retratos}
\end{figure}

\begin{figure}[h!]
\begin{center} 
\includegraphics[scale=0.6]{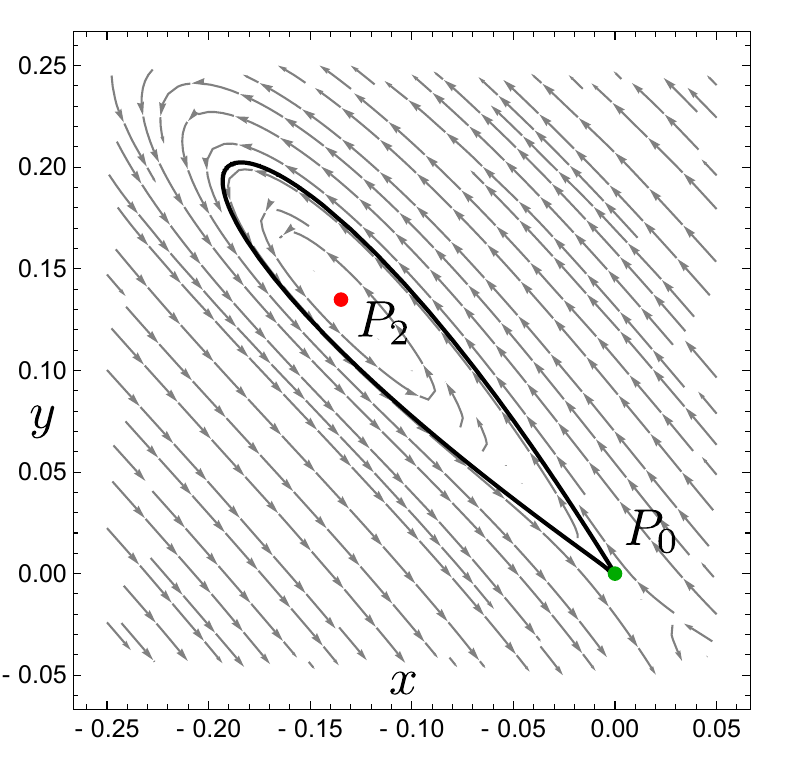} 
\end{center}
\caption{The phase portrait in the $(x,y)$-plane shows the homoclinic orbit of system \eqref{MSNcero} for the parameter values $p = 1.4431$, $r = 1.32431$, and $s = 1.2$. This trajectory corresponds to the orbit shown in Figure~\ref{retratos}(b), represented there in transformed coordinates. The equilibrium points $p_0$ and $p_2$ are shown
in green and red, respectively.}
\label{homoxy}
\end{figure}

We now turn to the parameter regimes defined by $\delta$ and $\mu$ that ensure the persistence of the homoclinic orbit $\Gamma_0^+$, which occurs, to first order in 
$\eta$, for the values of $\lambda$ given by
$$
\lambda = \lambda^+=\frac{5 \sqrt{2} \delta  \left(2 \delta ^2+9 \mu \right)^2\text{arctan}(\beta )+3 \sqrt{\mu } \left(10 \delta ^4+75 \delta ^2 \mu +108 \mu ^2\right)}{15 \left(\sqrt{2} \delta   \left(2 \delta ^2+9 \mu \right)\text{arctan}(\beta )+3 \sqrt{\mu } \left(\delta ^2+3 \mu \right)\right)},
$$
where $\beta=\frac{\sqrt{2}\delta+\sqrt{2\delta^2+9\mu}}{3\sqrt{\mu}}$. The corresponding graph is shown in Figure~\ref{persistencia3parder}.
\begin{figure}[hbt]
\centering
{\includegraphics[scale=0.6]{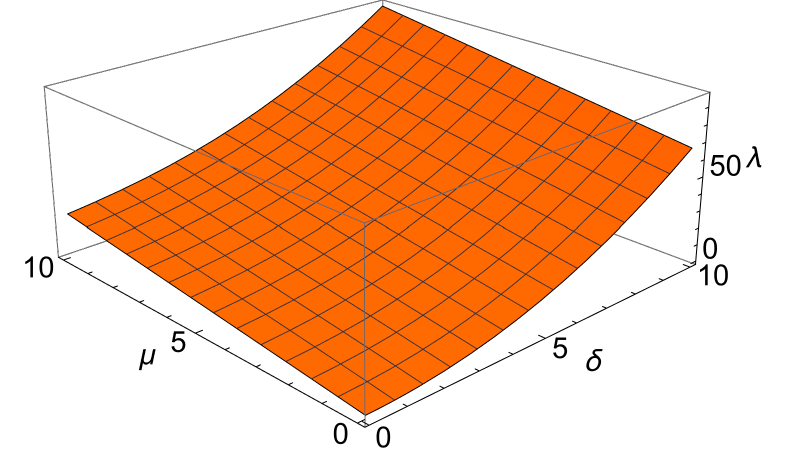}}
\caption{Plot of the surface $D^+(\delta,\mu,\lambda) = 0$, representing the set of parameter values for which the stable and unstable manifolds intersect 
 following the perturbation of the homoclinic orbit $\Gamma_0^+$, thereby indicating its persistence.}\label{persistencia3parder}
\end{figure}

As established in \cite{camassa}, Melnikov’s multi-pulse theory states that, to leading order, the Melnikov function associated with a double-loop homoclinic orbit can be approximated by the sum of the contributions from each individual single-loop connection. The passage near the saddle point contributes only higher-order correction terms. This approximation implies that the bifurcation condition $D^\pm(\delta, \mu, \lambda) = 0$, which determines the persistence of large-amplitude homoclinic connections to $p_0$, reduces to an algebraic combination of the single-pulse Melnikov functions. 

We distinguish between single-loop and double-loop homoclinic orbits. The single-loop orbits $\Gamma_0^\pm$ correspond to individual homoclinic connections associated with a single large excursion in phase space. In contrast, double-loop homoclinic configurations arise from the concatenation of two such elementary pulses, leading to more complex global trajectories. From a climate perspective, single-loop orbits can be interpreted as isolated large-amplitude transitions between quasi-stable climate states, whereas double-loop structures represent sequences of such transitions, capturing more intricate patterns of glacial--interglacial variability.

As a result, this formulation provides a practical method for identifying parameter regimes in which the double-loop orbit persists under small dissipative perturbations characterized by the system parameters.

In our case, the parameter values $\mu$, $\delta$, and $\lambda$ for which the homoclinic orbit persists under small perturbations are determined, to first order in $\eta$, by
\begin{align*}
\lambda^\cup (\mu,\delta) &= \dfrac{3I_2^- -2\delta I_1^- +
3I_2^+ -2\delta I_1^+
}{I_0^- + I_0^+}\\
& = \frac{2 \left(5 \sqrt{2} A \delta  \left(2 \delta ^2+3 \mu \right) \left(2 \delta ^2+9 \mu \right)+6 \sqrt{\mu } \left(10 \delta ^4+45 \delta ^2 \mu +18 \mu ^2\right)\right)}{5 \sqrt{2} A \delta  \left(2 \delta ^2+9 \mu \right)+30 \sqrt{\mu } \left(\delta ^2+3 \mu \right)},
\end{align*}
 where $A=\text{arctan}\left(\frac{\sqrt{2}\delta +\sqrt{2\delta^2+9\mu}}{3\sqrt{\mu}}\right)+\text{arctan}\left(\frac{\sqrt{2}\delta -\sqrt{2\delta^2+9\mu}}{3\sqrt{\mu}}\right)$. This surface is shown in Figure \ref{homo2doble}.

\begin{figure}[h!]
\centering
{\includegraphics[scale=0.5]{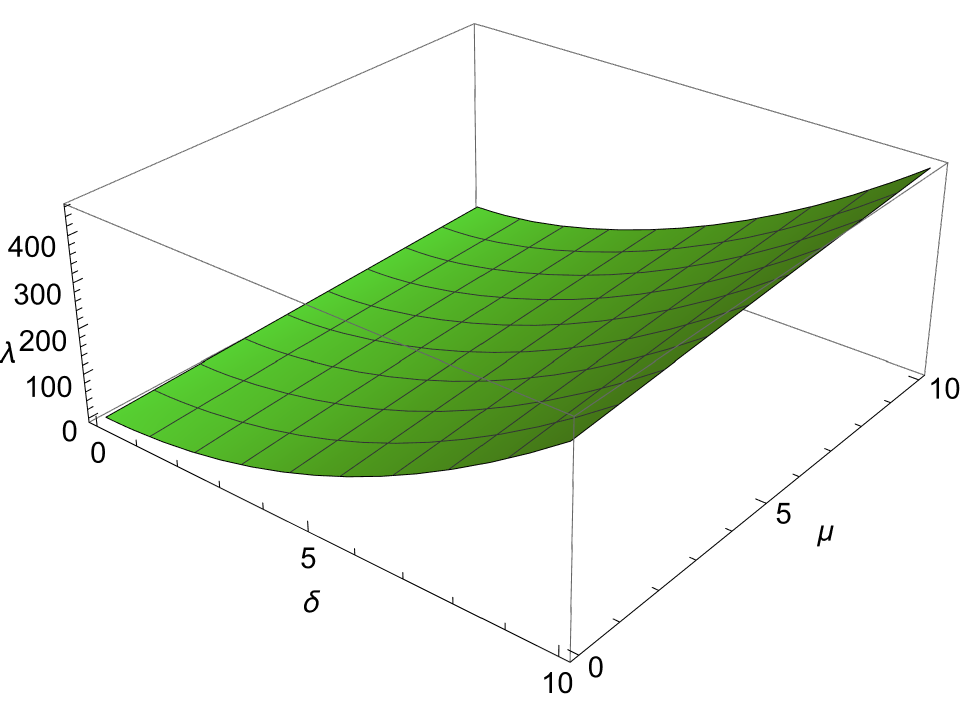}}
\caption{Plot of the surface $\lambda^{\cup}(\mu, \delta)$ indicating parameter regimes where double-loop homoclinic orbits to the origin persist, as predicted by the Melnikov function.}
\label{homo2doble}
\end{figure}

\begin{figure}[h!]
\centering
\hspace{0.4cm}
\subfigure[\; $\eta=5\times 10^{-8}$]{
\includegraphics[scale=0.5]{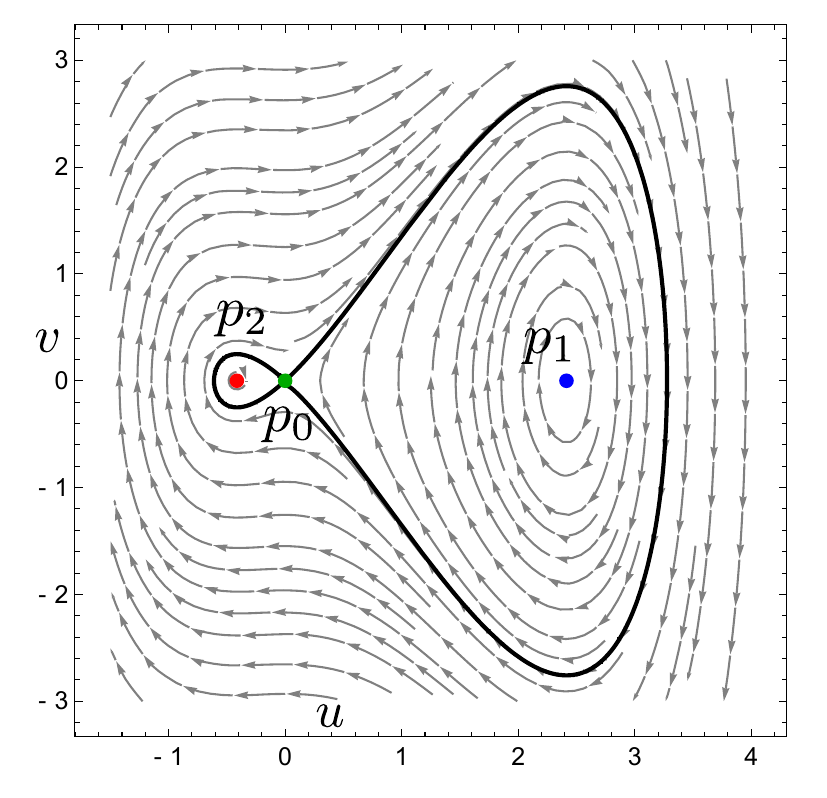}}
\hspace{0.4cm}
\subfigure[\; $\eta=5\times 10^{-7}$]{
\includegraphics[scale=0.5]{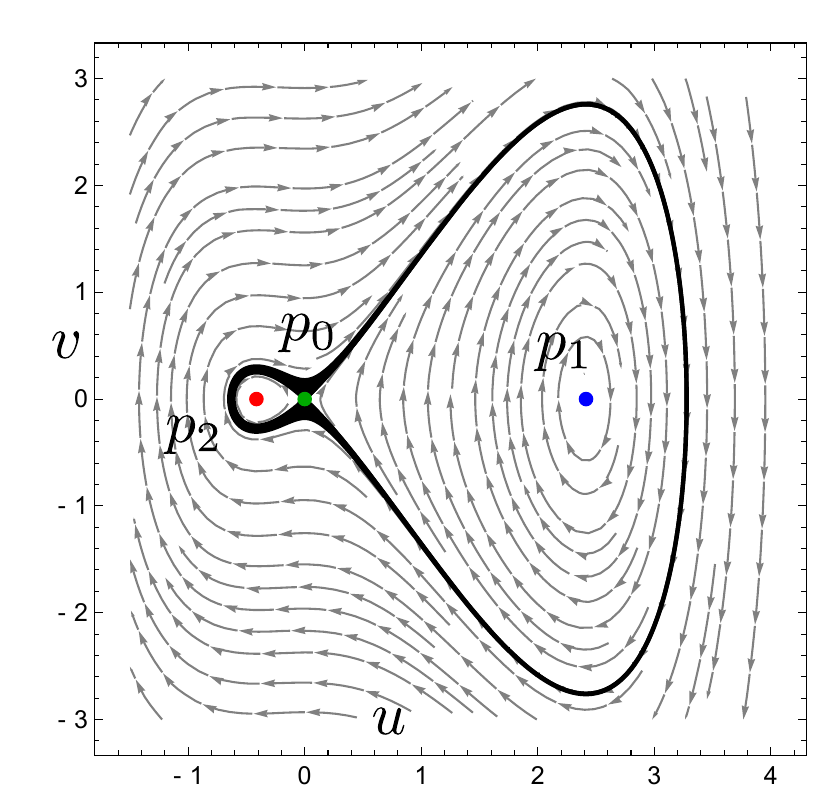}}
\caption{(a) Phase portrait of system \eqref{hamhomo} in the $(u,v)$-plane showing the persistence of the a double-loop homoclinic orbit for the parameter values $\delta = 2$, $\mu = 1$, and $\lambda = 21.1719$.
(b) The homoclinic orbit breaks under a perturbation in the parameter $\eta$, revealing the system's sensitivity to changes in this parameter.}
\label{homodoblepers}
\end{figure}

\section{Asymmetric Homoclinic Orbits for $s \neq 0$}\label{sub_sn0}
When the symmetry condition $s = 0$ is relaxed, the reflection invariance of the system is
broken and the dynamics become asymmetric. In this case, the invariant manifolds of the
saddle equilibria no longer coincide, and the homoclinic connections are deformed, giving rise
to a single asymmetric loop. Although the Hamiltonian structure is lost, numerical
continuation shows that a perturbed homoclinic trajectory persists for moderate values of $s$.
This orbit encloses one of the nontrivial equilibria and organizes the nearby phase portrait,
producing large-amplitude oscillations that connect distinct regions of the slow manifold.

From a dynamical perspective, the persistence of this asymmetric homoclinic orbit indicates
that global connections continue to govern transitions between coexisting quasi–steady regimes.
Within the anomaly framework adopted in this work, these transitions represent large
excursions in the system's state variables rather than literal physical changes such as ice
loss or accumulation. This interpretation preserves the physical consistency of the model
while emphasizing the geometric mechanisms responsible for the organization of the global
phase space.

\begin{figure}[ht!]
\begin{center} 
\subfigure[\; $p = 1.0165$, $r = 1.0215$, $s=0.1$]{
\includegraphics[scale=0.5]{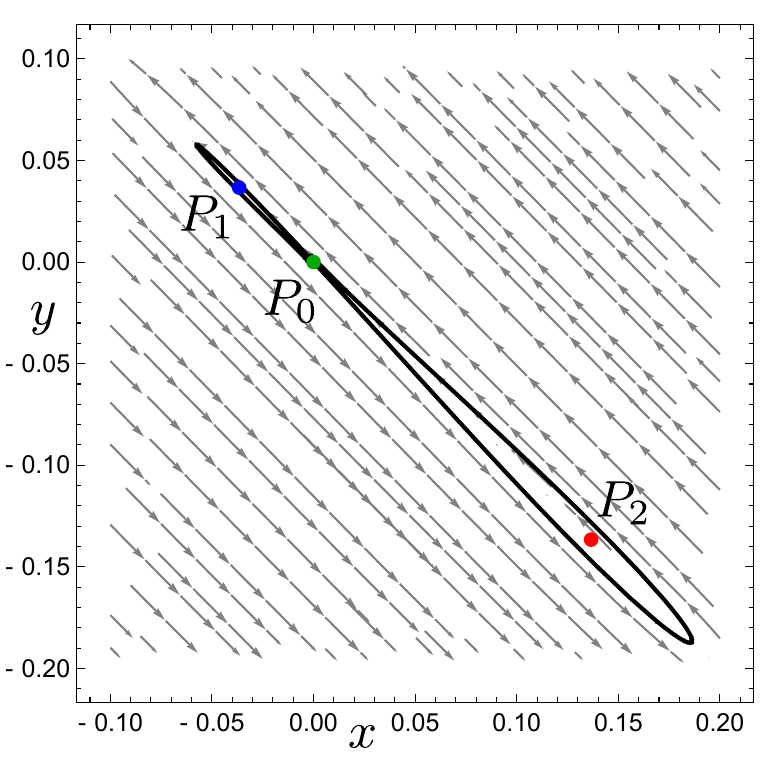}}\hspace{0.3cm}
\hspace{0.4cm}
\subfigure[\; $p = 1.059$, $r = 1.079$, $s=0.2$]{
\includegraphics[scale=0.5]{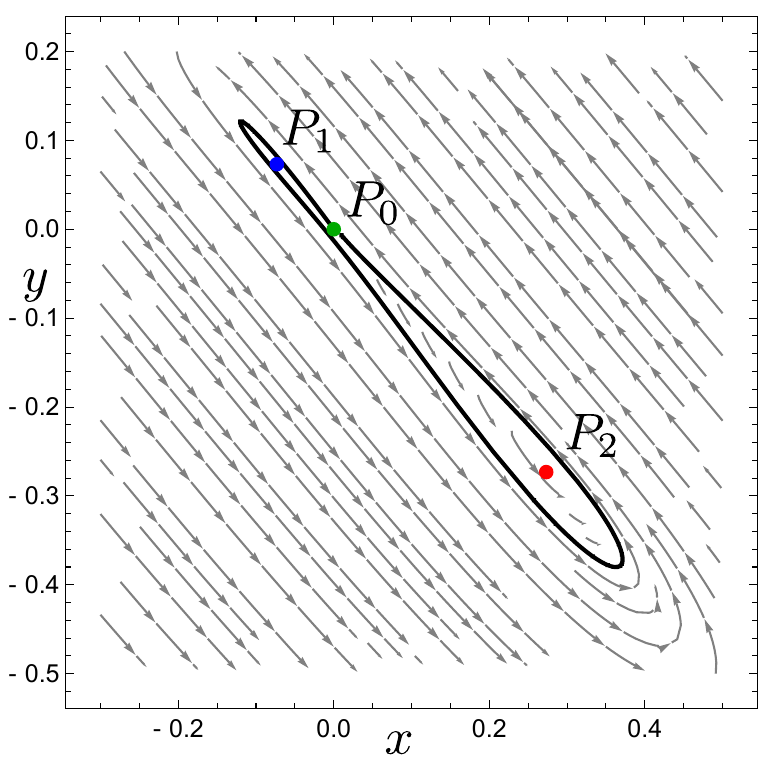}}
\hspace{0.5cm}
\subfigure[\; $p = 1.24695$, $r = 1.32695$, $s=0.4$]{
\includegraphics[scale=0.5]{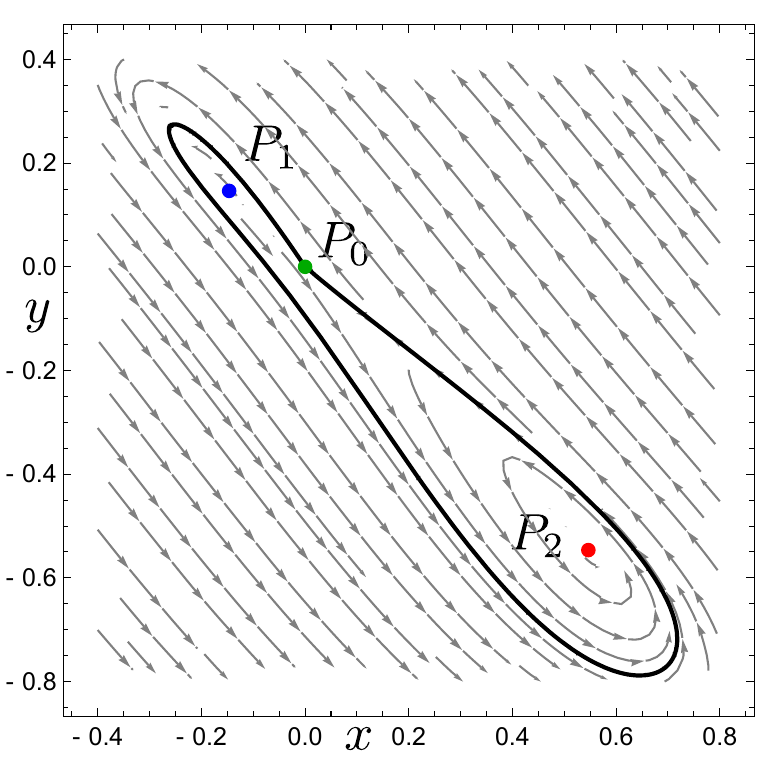}}
\hspace{0.6cm}
\subfigure[\; $p = 1.614834895$, $r= 1.794834895$, $s=0.6$]{
\includegraphics[scale=0.5]{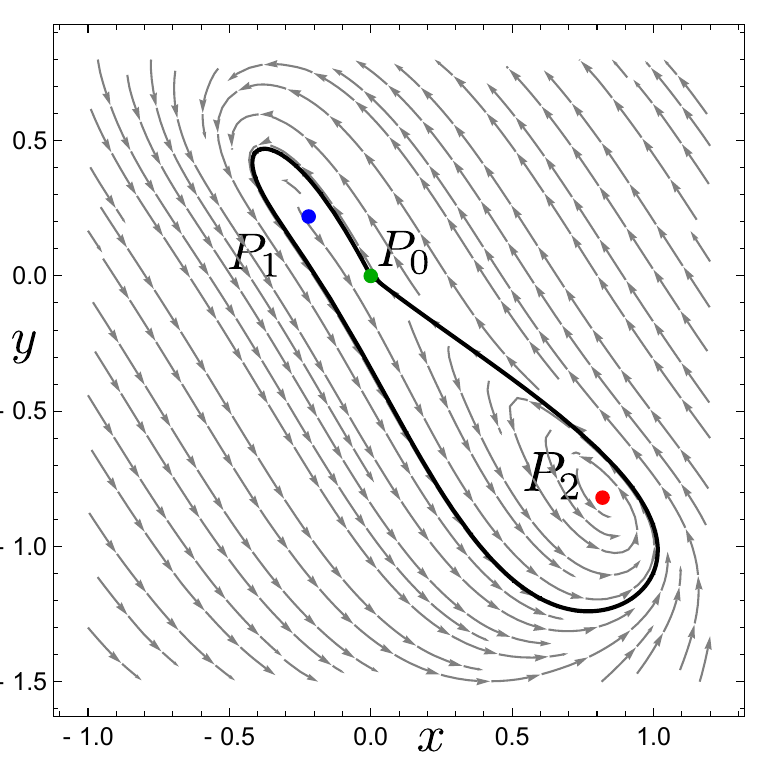}}
\end{center}
\caption{Phase portraits of the  system \eqref{MSNcero} in the $(x,y)$-plane  for representative values of the parameters $p$, $r$ and $s=0.1,0.2,0.4$ and $0.6$ illustrating  the existence of large-amplitude homoclinic orbits. The equilibrium points $P_0$, $P_1$, and $P_2$ are shown in green, blue, and red, respectively.}\label{fig-per-dob}
\end{figure}

In view of the preceding analysis, we now examine the persistence and bifurcation structure of homoclinic orbits in the singular limit $s \to 0$, where the system becomes $\mathbb{Z}_2$-symmetric. While this limit does not correspond to a physically meaningful configuration, it serves as a mathematically well-posed approximation that reveals the organizing structure of the dynamics.

\section{Symmetric homoclinic bifurcation in the limit $s = 0$}\label{sub_s0}
We now focus on the symmetric case $s = 0$.
This regime offers a simplified yet insightful framework in which the existence and structure of symmetric homoclinic connections can be analyzed rigorously. In particular, the presence of $\mathbb{Z}_2$ symmetry facilitates a more tractable analytical approach and highlights the role of symmetry in shaping the global phase portrait.

By taking the limit $s \to 0$ in system~\eqref{homo}, we obtain the reduced system:
\begin{equation}
\label{homo1}
\begin{split}
\dot{x} &= y,\\
\dot{y} &= (r - p)\,x + (r - 1)\,y - (x + y)^3,
\end{split}
\end{equation}

System \eqref{hamhomo} then transforms into
\begin{equation}
\label{homo2}
\begin{split}
\dot{u} &= v,\\
\dot{v} &= \mu u - u^3 + \eta(\lambda v - 3u^2v) - 3\eta^2 uv^2 - \eta^3 v^3,
\end{split}
\end{equation}
which becomes the Hamiltonian system
\begin{equation}
\label{homo3}
\begin{split}
\dot{u} &= v,\\
\dot{v} &= u - u^3,
\end{split}
\end{equation}
without loss of generality, we assume $\mu=1$ and set $\eta=0$, with associated Hamiltonian function
\begin{equation}
\label{ham1}   
H(u,v) = \frac{1}{2}v^2 - \frac{1}{2}u^2 + \frac{1}{4}u^4.
\end{equation}

The equilibrium points \( P_{1,2} = (\pm 1, 0) \) are local minima of \( H(u,v) \), whereas \( P_0 = (0, 0) \) is a saddle point. On the level curve defined by \( H_0 = H(u,v) = 0 \), there exists a pair of symmetric homoclinic orbits \( \Gamma^{\pm}_0 \),
$$
\Gamma^{\pm}_0:\: \left(u_0^{\pm}(\tilde{t}), v_0^{\pm}(\tilde{t})\right) = \left(\pm \sqrt{2}\,\text{sech}(\tilde{t}), \mp\,\text{sech}(\tilde{t})\,\text{tanh}(\tilde{t})\right).
$$

We now consider the perturbed Hamiltonian system \eqref{homo2}, written as
\begin{equation}\label{Hamhomo}
\begin{pmatrix}
\dot{u}\\
\dot{v}
\end{pmatrix}
=
\begin{pmatrix}
v\\
u - u^3
\end{pmatrix}
+ \eta
\begin{pmatrix}
0\\
\lambda v - 3u^2v
\end{pmatrix}
+ \mathcal{O}(\eta^2).
\end{equation}

Using the Melnikov method, we obtain the asymptotic expansion
\begin{equation}
\label{melnikov}
D^{\pm}(\lambda) = \eta M(\lambda) + \mathcal{O}(\eta^2),
\end{equation}
where the Melnikov function is given by
\[
M^{\pm}(\lambda) = \int_{-\infty}^{\infty} \left(\lambda - 3(u_0^{\pm}(t))^2\right)\left(v_0^{\pm}(t)\right)^2 d\tilde{t}.
\]
Evaluating this expression, we find that \( M^{\pm}(\lambda) = 0 \) when
\[
\lambda = \frac{3 \int_{-\infty}^{\infty} (u_0^{\pm}(\tilde{t})v_0^{\pm}(\tilde{t}))^2 d\tilde{t}}{\int_{-\infty}^{\infty} (v_0^{\pm}(\tilde{t}))^2 d\tilde{t}} = \frac{12}{5}.
\]

The persistence of double-loop homoclinic orbits is illustrated in Figures \ref{retratosdoble} and \ref{homoxydoble} in the $(u, v)$ and $(x, y)$ coordinates, respectively, for parameter values $\delta = 0$, $\mu = 1$, and $\lambda = 2.40378$, which correspond to $s = 0$, $p = 1.01404$, and $r = 1.02404$.
\begin{figure}[h!]
\centering
\subfigure[\; $\eta=0$]{
\includegraphics[scale=0.5]{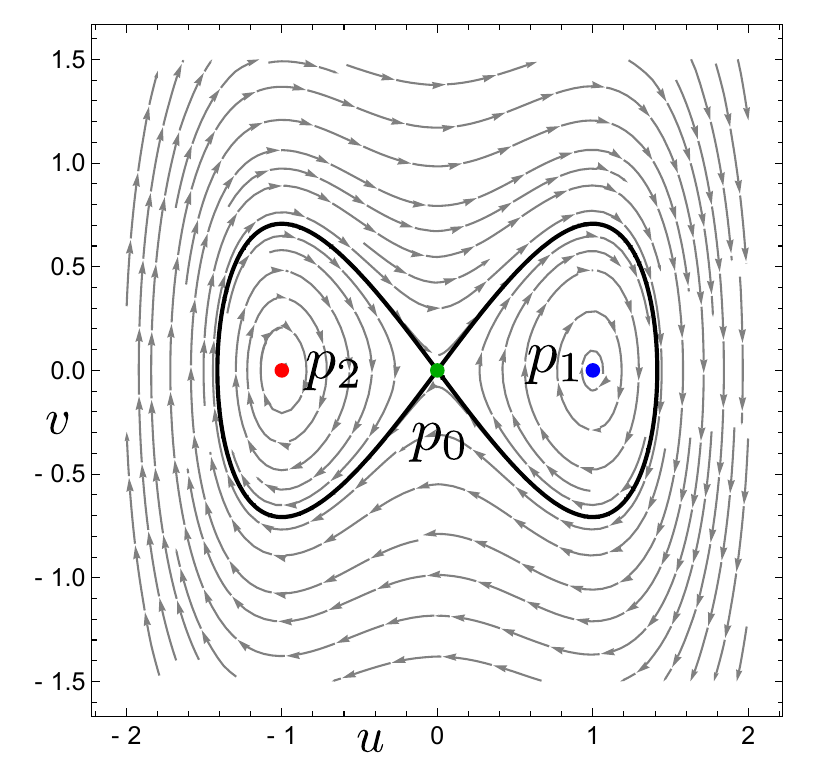}\hspace{0.8cm}}
\subfigure[ \; $\eta=0.1$]{
\includegraphics[scale=0.5]{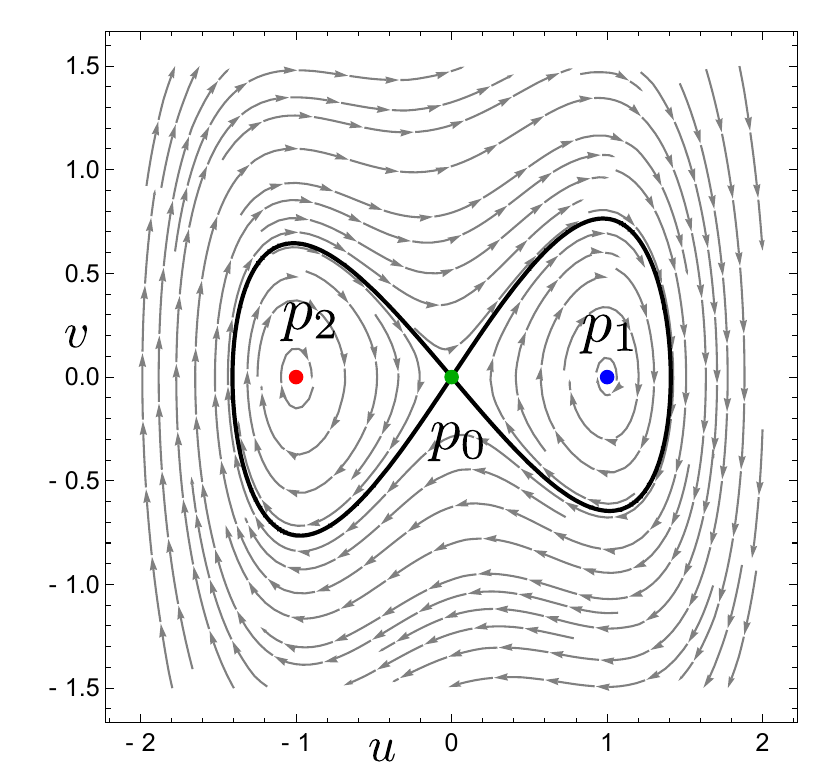}}
\caption{Double-loop homoclinic orbits in the $(u,v)$-plane for parameters $\delta = 0$, $\mu = 1$, and $\lambda = 2.40378$: (a) shows the orbit in the unperturbed Hamiltonian system~\eqref{Ham}, while (b) illustrates its persistence in the perturbed system~\eqref{hamhomo}.}\label{retratosdoble}
\end{figure}

\begin{figure}[h!]
\begin{center} 
\includegraphics[scale=0.5]{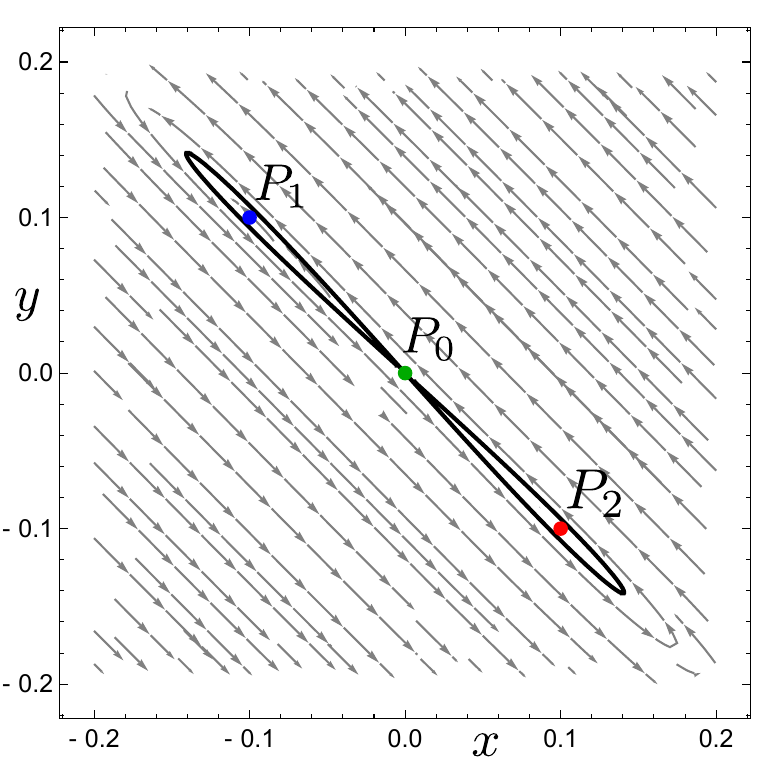} 
\end{center}
\caption{Phase portrait in the $(x, y)$-plane showing a double-loop homoclinic orbit of system~\eqref{MSNcero} for the parameter values $p =  1.01404$, $r = 1.02404$, and $s = 0$.}
\label{homoxydoble}
\end{figure}

In addition to the perturbation of the homoclinic orbits $\Gamma_0^\pm$, we must also consider the persistence of closed level curves of the hamiltonian $H$ lying inside and outside the homoclinic loop.

Let $\gamma^\alpha(\tilde{t}) = (u_\alpha(\tilde{t}), v_\alpha(\tilde{t}))$ denote a periodic orbit lying inside the homoclinic loop, associated with the level set \( H(u_\alpha, v_\alpha) = \alpha \), and let $T_\alpha$ be its period.

Choosing a level curve \( \gamma^\alpha = H^{-1}(\alpha) \), we define the corresponding Melnikov function as
\begin{align*}
M^\alpha(\lambda) &= \int_{0}^{T_{\alpha}} \left(\lambda - 3u_{\alpha}^2(t)\right) v_{\alpha}^2(t)\, d\tilde{t} \\
&= \lambda \int_{\gamma^{\alpha}} v\, du - 3 \int_{\gamma^{\alpha}} u^2 v\, du,
\end{align*}
where we have used $v = du/dt$ to convert the time integral into a line integral along the closed orbit $\gamma^\alpha$. Using the Hamiltonian \eqref{ham1} to express $v$ as a function of $u$, we find that the condition for persistence of $ \gamma^\alpha$ is $M^\alpha(\lambda) = 0$, or 
\begin{equation}
\label{Ralpha}
\lambda = \frac{3 \int_{\gamma^{\alpha}} u^2 v\, du}{\int_{\gamma^{\alpha}} v\, du} = \mathcal{R}(\alpha).
\end{equation}

As shown in \cite{Carr}, the function $\mathcal{R}(\alpha)$  has the qualitative shape illustrated in Figure~\ref{figCarr}: it attains a unique minimum at $\kappa \approx 2.256$, and then increases monotonically as $\alpha \to \infty$. Hence, for $\lambda \in (3, \infty)$, only one closed orbit near a level set persists.
For certain values of $\alpha > 0$, a single closed orbit is preserved. When $\lambda \in (\tfrac{12}{5}, 3)$, three periodic orbits coexist. For $\lambda \in (\kappa, \tfrac{12}{5})$, two of these orbits remain, while one periodic orbit persists throughout the interval $ \alpha > 0$. At the critical value $ \lambda= \tfrac{12}{5}$, a homoclinic saddle connection appears, as previously described, coexisting with a periodic orbit located outside the homoclinic loop. In the range $ \lambda\in (\kappa, \tfrac{12}{5})$, the system exhibits two nested periodic orbits: an inner repelling cycle enclosed by an outer attracting one, both surrounding the three equilibrium points. As $\lambda$ passes through $\kappa$, these two orbits collide and vanish in a saddle-node bifurcation of periodic orbits.

As $s \to 0$, the dashed line in Figure~\ref{centro_org1} collapses onto the pitchfork bifurcation line, and the organizing centers $Q_0$ and $Q_1$ merge into the $\mathbb{Z}_2$-symmetric Bogdanov--Takens point located at $Q = (1, 1)$. 

To facilitate the visualization of the bifurcation structure in parameter space, we fix $\mu = 1$ without loss of generality and observe that $\lambda$ remains positive along both homoclinic branches. In this symmetric limit, the Hopf bifurcation curves coalesce into a single branch, while the homoclinic bifurcation curves converge to a single curve that is tangent to the line $r - 1 = \frac{12}{7}(p - 1)$ at $Q$. Likewise, the saddle--node bifurcation curves of limit cycles collapse into a single curve tangent to the line $r - 1 \approx 1.7962(p - 1)$ at the same point.

Figure \ref{diagramaC} depicts the global bifurcation structure with $\mu = 1$, emphasizing the role of the organizing center and the bifurcation curves emerging from it.

\begin{figure}[hbt]
\begin{center} 
\includegraphics[scale=0.6]{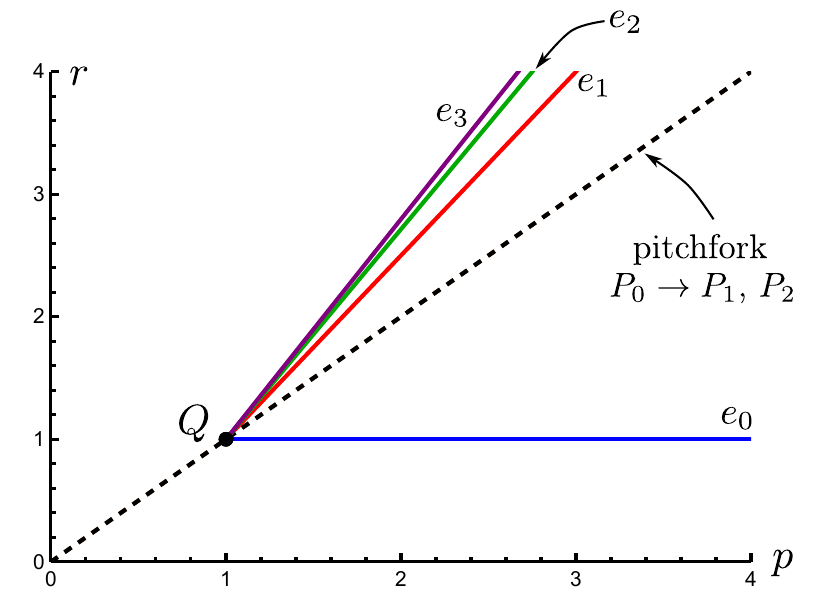} 
\end{center}
\caption{Bifurcation diagram of system~\eqref{MSNcero} for the symmetric case $s=0$. The point $Q=(1,1)$ represents the unique organizing center of the system. The bifurcation curves are as follows: the dashed line corresponds to a pitchfork bifurcation curve along which the equilibria $P_1$ and $P_2$ emerge from the origin; $e_0$ (blue) represents a supercritical Hopf bifurcation at the origin $P_0$, while $e_1$ (red) corresponds to a subcritical Hopf bifurcation at the equilibria $P_1$ and $P_2$. The curve $e_2$ (green) indicates a homoclinic bifurcation associated with the saddle at the origin, and $e_3$ (purple) corresponds to a saddle-node bifurcation curve along which equilibria are created or annihilated near the origin.}\label{diagramaC}
\end{figure}

To provide numerical support for the analytical results shown in Figure~\ref{figCarr}, we present in Figure~\ref{homosuv}  phase portraits of system~\eqref{homo2} in the $(u, v)$-plane for $\mu = 1$ and several representative values of the parameter $\lambda$. These plots illustrate the evolution of the dynamics, including the persistence and breakdown of homoclinic structures. The equilibrium points $p_0$, $p_1$, and $p_2$ are marked in green, blue, and red, respectively. The observed transitions in the phase portraits align with the Melnikov predictions and reflect the bifurcation scenario depicted in Figure~\ref{figCarr}.
\begin{figure}[h!]
\centering
\includegraphics[scale=0.6]{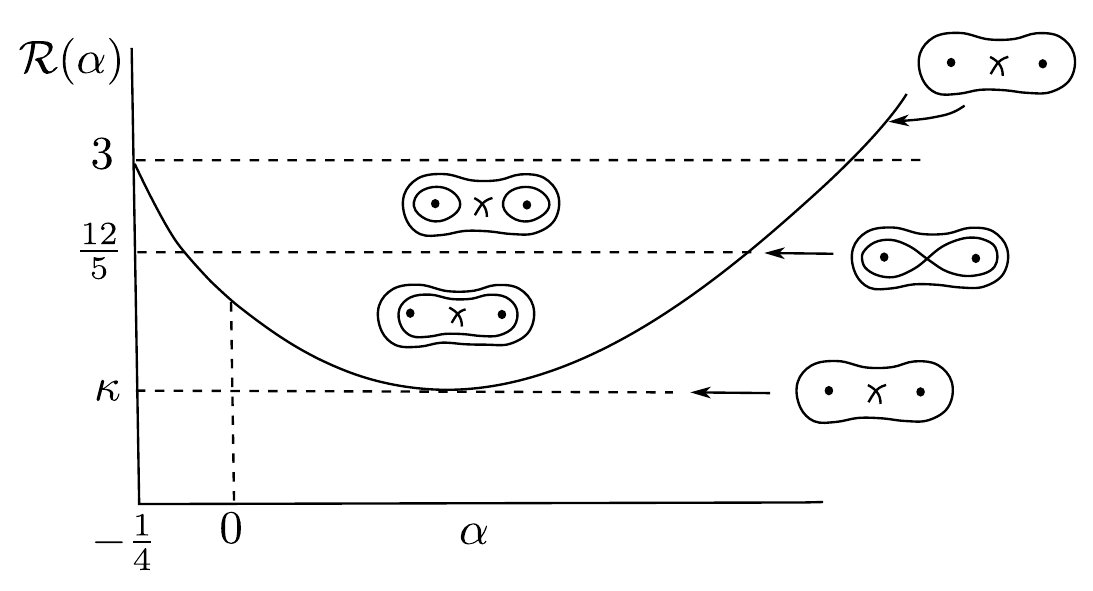}
\caption{The graph of ${\mathcal R}(\alpha)$ together with its associated invariant level curves, where  $\kappa \approx  2.256$.}
\label{figCarr}
\end{figure}
\begin{figure}[!htbp]
\begin{center} 
\subfigure[\; $\lambda = 2$]{
\includegraphics[width=5.5cm]{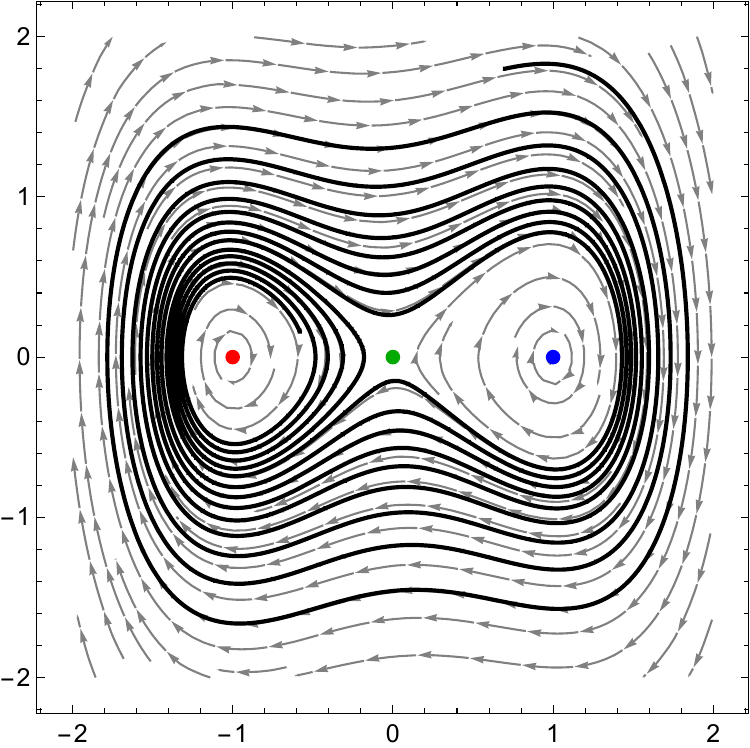}}\hspace{0.3cm}
\hspace{0.4cm}
\subfigure[\; $\lambda = 2.256$]{
\includegraphics[width=5.5cm]{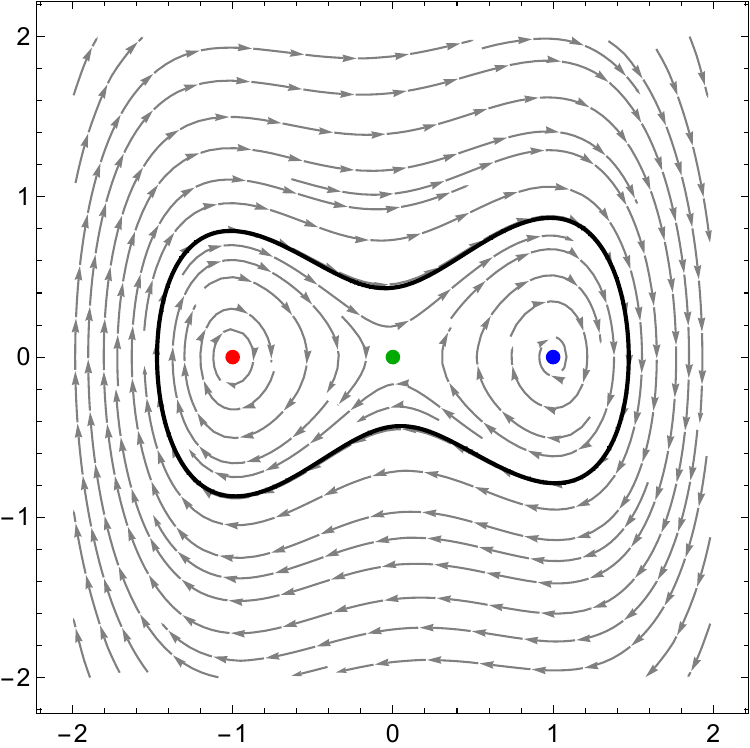}}\hspace{0.3cm}
\hspace{0.4cm}
\subfigure[\; $\lambda=2.3$]{
\includegraphics[width=5.5cm]{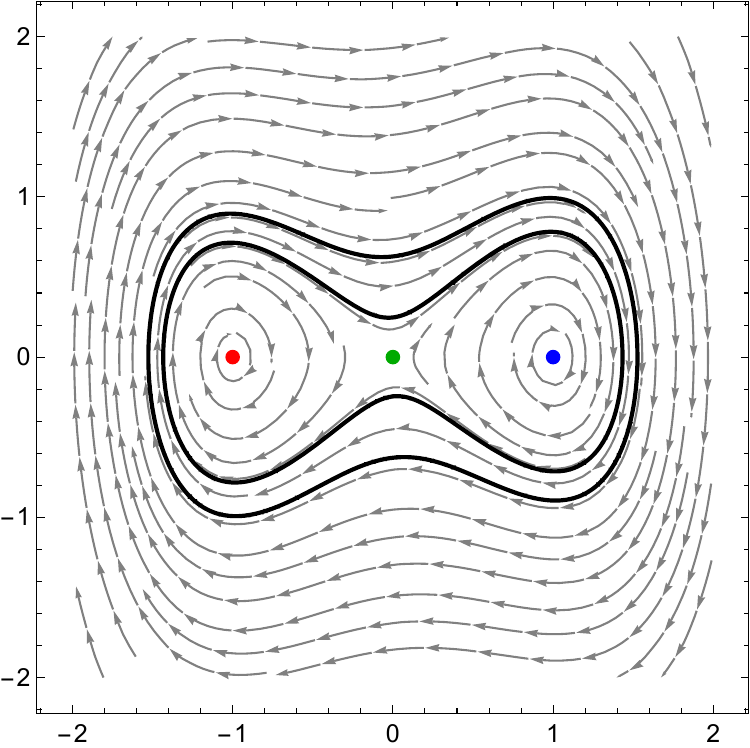}}
\hspace{0.5cm}
\subfigure[\; $\lambda=2.4$]{
\includegraphics[width=5.5cm]{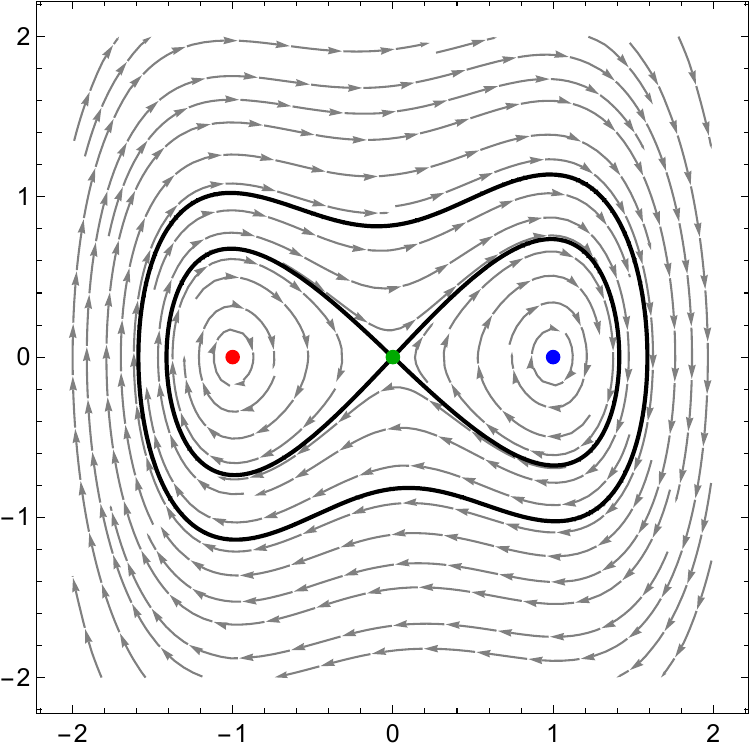}}
\hspace{0.6cm}
\subfigure[\; $\lambda=2.6$]{
\includegraphics[width=5.5cm]{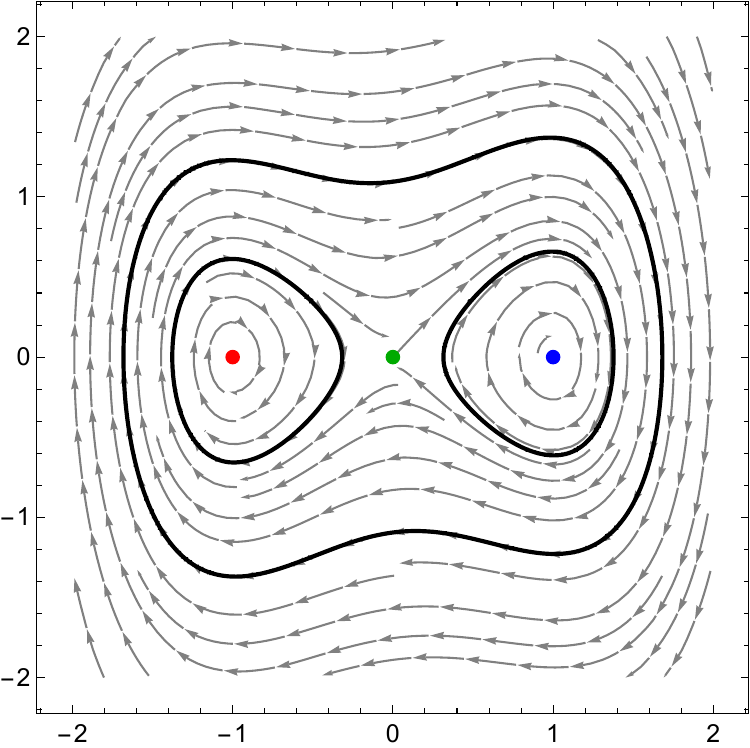}}
\hspace{0.6cm}
\subfigure[\; $\lambda=3.2$]{
\includegraphics[width=6cm]{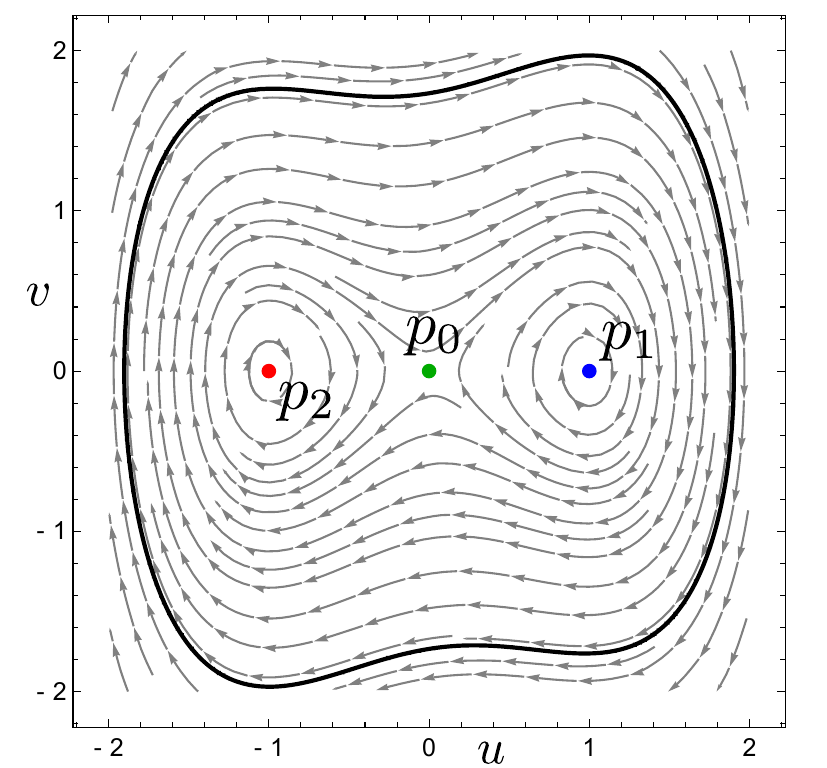}}
\end{center}
\caption{Phase portraits of the Hamiltonian system \eqref{homo2} in the $(u,v)$-plane illustrating different dynamical scenarios for $\mu=1$ and representative values of the parameter $\lambda$. 
In (a) there are no periodic orbits; in (b) there is a single periodic orbit; in (c) two periodic orbits are present; in (d) a homoclinic orbit and an enclosing periodic orbit are observed; in (e) an envelope and a periodic orbit surround each of the equilibrium points located away from the origin; and in (f) only one periodic orbit exists. 
In all panels, from left to right, the equilibrium points are $p_2$ (red), $p_0$ (green), and $p_1$ (blue).}\label{homosuv}
\end{figure}

\section{Concluding Remarks}\label{sec_conclusion}
This work provides a rigorous analytical refinement of the deterministic Saltzman-Maasch
model \eqref{MSN}, formulated within a slow-fast dynamical framework. By incorporating a cubic
nonlinear feedback in the CO$_2$ dynamics, we derived explicit criteria for the occurrence
and criticality of a generalized Hopf bifurcation. While this bifurcation had been
previously detected through numerical continuation by  Nolet \cite{Nolet2017} and
Quinn \cite{Quinn2018}, our analysis establishes it analytically for the first time by computing
the first and second Lyapunov coefficients and identifying the precise transition between
subcritical and supercritical Hopf regimes. This result provides a solid mathematical
foundation for the coexistence of stable and unstable limit cycles in the model.

A second contribution concerns the global organization of the dynamics. Through a
Hamiltonian reduction and the application of Melnikov’s method, we derived explicit
conditions ensuring the persistence of homoclinic orbits under small perturbations. These
global connections between saddle equilibria generate large-amplitude excursions in the
phase space and correspond to abrupt transitions between quasi-steady climate states.
Their analytical characterization confirms that deterministic internal feedbacks are
sufficient to reproduce glacial terminations without invoking stochastic forcing.

The dynamical regimes identified in this work can be related to key features of Pleistocene glacial cycles. In particular, the emergence of stable and unstable limit cycles provides a mechanism for oscillations with different amplitudes and periods, while homoclinic connections are associated with abrupt transitions between climate states. These features are consistent with paleoclimate records, which exhibit variability in both amplitude and duration, as well as rapid glacial--interglacial transitions.

From a broader perspective, several extensions of the present framework are of interest. In particular, the inclusion of stochastic forcing may help explain transitions between climate states that are not captured by the deterministic model alone, as observed in related stochastic versions of the Saltzman--Maasch model (see, e.g., \cite{Alexandrov}). Likewise, incorporating time-dependent orbital forcing could provide a more direct link with Milankovitch cycles, while higher-dimensional extensions may allow for the inclusion of additional feedback mechanisms relevant to long-term climate variability.

Overall, the present analysis unifies local (Hopf and generalized Hopf) and global (homoclinic)
bifurcations within a single geometric framework, showing how nonlinear feedbacks and
slow-fast interactions can organize multiple oscillatory regimes. These results support the
interpretation of Pleistocene climate variability as an emergent property of intrinsic
deterministic dynamics, while also highlighting the potential role of external and stochastic
effects in shaping observed climate transitions.
\section*{Appendices}

\appendix
\section{Proof of Proposition~\ref{prop_bif1}}\label{ap_lyap}
The Lyapunov coefficients are computed using the normal form formulas given in the book by Kuznetsov \cite{Kuznetsov2023}.
\medskip

We start by considering  that the Jacobian matrix of the system  \eqref{MSNcero} evaluated in $r=1$,  which is given by
\begin{equation*}
{\mathcal A}=\begin{pmatrix}
-1& -1\\
p & 1\\
\end{pmatrix},
\end{equation*}
with  characteristic polynomial  $p(\lambda)=\lambda^2+p-1$. 
Therefore, if  $p>1$, the matrix ${\mathcal A}$  has a pair of purely complex
imaginary eigenvalues
$\lambda_{1,2}=\pm i \omega$, where  $\omega = \sqrt{p-1}$.
Now, we verify the transversality condition of the Hopf bifurcation. A simple calculation shows that
$$
\frac{1}{2}\frac{\partial \: {\rm tr} ({\mathcal A})}{\partial r}\Big|_{r=1}  
= \frac{1}{2}\frac{\partial (r-1)}{\partial r}\Big|_{r=1}  = \frac{1}{2}\neq 0.
$$
This confirms that the transversality condition holds. Consequently, system \eqref{MSNcero} undergoes a Hopf bifurcation at the origin when the parameter crosses the threshold value $r = 1$.

To investigate the nature of the Hopf bifurcation and the stability of the periodic orbits it generates, we now compute the first Lyapunov coefficient $l_1$. This coefficient provides a rigorous criterion for determining whether the bifurcation is supercritical or subcritical, and hence whether the emerging limit cycle 
 from the equilibrium point $(0,0)$
is stable or unstable.

The two complex eigenvectors associated with the eigenvalues $\pm i\omega$ are chosen as
$$
{\mathbf q} = \left( \frac{\omega +i}{2\omega} , -\frac{(1+\omega^2)i}{2\omega}\right)^T, \qquad 
{\mathbf p} = \left( 1 , \frac{1-i\omega}{1+\omega^2}\right)^T,
$$
where $\omega=\sqrt{p-1}$. 
They are normalized such that $\langle \mathbf p, \mathbf q \rangle = 1$, and the multilinear terms and normal form coefficients $g_{20}$, $g_{11}$, and $g_{21}$ are computed using the formulas for Hopf bifurcations given in \cite{Kuznetsov2023}. We obtain
\[
g_{20}= \frac{(1+i\sqrt{p-1})ps}{2(p-1)}, 
\quad 
g_{11}= -g_{20}, 
\quad  
g_{21}=-\frac{3p^2(-i+\sqrt{p-1})}{4(p-1)^{3/2}}.
\]
Using the standard expression for the first Lyapunov coefficient, we obtain
\begin{equation}\label{l1}
l_1= -\frac{p(3(p-1)-2s^2)}{8(p-1)^{5/2}}.
\end{equation}
Therefore, $l_1 >0$ when $1<p<1+\tfrac{2}{3}s^2$, while $l_1 <0$ holds for $p>1+\tfrac{2}{3}s^2$.

It is worth noting that $l_1=0$ at $p = 1+ \tfrac{2}{3}s^2$, implying that a generalized Hopf bifurcation (Bautin bifurcation) may occur.
To better understand the bifurcation, we need to calculate the second Lyapunov coefficient $l_2$. 

Notice that for $p = 1+ \tfrac{2}{3}s^2$, the matrix ${\mathcal A}$ has eigenvalues 
 $\lambda_{1,2}= \pm i \sqrt{\tfrac{2}{3}}s$. 
In the meantime, to compute $l_2$, we begin by calculating the normalized eigenvector
 of   ${\mathcal A}$ and the adjoint eigenvector  of ${\mathcal A}^T$, namely
 $${\mathbf q} =  \left( \begin{array}{c}\tfrac{1}{2}+ i\tfrac{\sqrt{6}}{4s}\vspace{0.2cm} \\ -i\tfrac{3+2s^2}{2\sqrt{6}s} \end{array} \right)
  \qquad \text{and}\qquad  
{\mathbf p} = \left( \begin{array}{c}1 \\ \tfrac{3-\sqrt{6}s i}{3+2s^2} \end{array} \right),$$
 respectively.
In addition, we  compute the quantities $g_{ij}$, $2\leq i+j\leq 4$, according to the  formulas  given in \cite{Kuznetsov2023}:
\begin{align*}
 & g_{20} = \frac{(3+i\sqrt{6} s)(3+2s^2)}{12s}, \quad g_{30} = \frac{(2s-i\sqrt{6})(3+2s^2)^2}{16s^3}, \quad
  g_{11} = -g_{20},  \\
 & g_{02} = g_{20}, \quad g_{21} = -g_{30} , \quad g_{12} = -g_{30}, \quad   g_{03} = -g_{30}\quad g_{40}=g_{31}=g_{13}=g_{22}=0.
\end{align*}
Based on these results, we calculate the second Lyapunov coefficient, and it follows easily that
\begin{equation}
\label{lipasegundo}
l_2=-\dfrac{5(3+2s^2)^4}{128\sqrt{6}s^7}, \end{equation}
which is clearly  negative for all $s>0$. Thus, 
the system \eqref{MSNcero} exhibits a supercritical generalized Hopf bifurcation at $(0,0)$ when $p = 1 + \frac{2}{3}s^2$, which makes $(0,0)$ a weakly stable focus.  This implies that two limit cycles bifurcate from
$(0,0)$: an inner, unstable limit cycle, and an outer, stable limit cycle (see \cite{Kuznetsov2023}).

From \eqref{val_cero}, the Jacobian matrix \eqref{matjp} has  conjugate complex eigenvalues 
$$\lambda_{1,2}= \mu(r,p)\pm i \omega(r,p) = \frac{1}{2}(r-1)\pm \frac{1}{2} i \sqrt{4p-(1+r)^2},$$
for all $0 < r \leq 2\sqrt{p}-1$ with $p>\frac{1}{4}$.

Next, we verify the transversality by checking the regularity of the map $\phi: (r,p)\mapsto  (\mu(r,p),l_1(r,p))$  at the point $(r,p)=(1,1+\frac{2}{3}s^2)$. We note that the Jacobian determinant of the map $\phi$
evaluated at $(r,p)=(1,1+\frac{2}{3}s^2)$ is
\begin{align*}
\left|\begin{array}{cc} \frac{\partial\mu(r,p)}{\partial r} &  \frac{\partial\mu(r,p)}{\partial p} \vspace{0.3cm}\\
\frac{\partial l_1(r,p)}{\partial r} &  \frac{\partial l_1(r,p)}{\partial p} \end{array}
 \right| \; _{\mathrel{\Bigg|}_{(r,p)=(1,1+\frac{2}{3}s^2)}}  & =
\frac{1}{2}  \frac{\partial l_1(r,p)}{\partial p} {\Bigg|}_{(r,p)=(1,1+\frac{2}{3}s^2)} \\
& =
\frac{3(p^2+p-2)-2(2+3p)s^2}{16(p-1)^{7/2}}{\Bigg|}_{(r,p)=(1,1+\frac{2}{3}s^2)}\\
& =-\frac{9\sqrt{\frac{2}{3}}(3+2s^2)}{32s^5}\neq 0,
\end{align*}
which means that the transversality condition is satisfied. 
We note that all the conditions for the generalized Hopf bifurcation are satisfied, leading to a codimension-2 generalized Hopf bifurcation at
$(r,p)=(1,1+\frac{2}{3}s^2)$.

\section{Computation of the Melnikov function}\label{ap_melnikov}
This appendix contains the computation of the Melnikov function associated with the homoclinic orbits.
Using \eqref{Hamm}, \eqref{Hamestr}, and the expressions for $\Gamma_0$ given in \eqref{gam1} and \eqref{gam2}, we obtain
\begin{equation*}
\begin{split}
M^\pm(\delta,\mu,\lambda)
&= \int_{-\infty}^{\infty}
\left.\left(\nabla H \cdot
\begin{pmatrix}
0\\
\mu u + \delta u^2 - u^3
\end{pmatrix}
\right)\right|_{\Gamma_0^\pm}
d\tilde{t} \\
&= \int_{-\infty}^{\infty}
\left.
\begin{pmatrix}
-\mu u - \delta u^2 + u^3\\
v
\end{pmatrix}
\cdot
\begin{pmatrix}
0\\
v(\lambda + 2\delta u - 3u^2)
\end{pmatrix}
\right|_{\Gamma_0^\pm}
d\tilde{t} \\
&= \int_{-\infty}^{\infty}
(v_0^\pm(\tilde{t}))^2
\Big(\lambda + 2\delta u_0^\pm(\tilde{t}) - 3(u_0^\pm(\tilde{t}))^2\Big)
\, d\tilde{t} \\
&= \lambda I_0^\pm(\delta,\mu) + 2\delta I_1^\pm(\delta,\mu) - 3 I_2^\pm(\delta,\mu),
\end{split}
\end{equation*}
where
\[
I_0^\pm(\delta,\mu) = \int_{-\infty}^{\infty} (v_0^\pm(\tilde{t}))^2 \, d\tilde{t}, \quad
I_1^\pm(\delta,\mu) = \int_{-\infty}^{\infty} u_0^\pm(\tilde{t}) (v_0^\pm(\tilde{t}))^2 \, d\tilde{t},
\]
\[
I_2^\pm(\delta,\mu) = \int_{-\infty}^{\infty} (u_0^\pm(\tilde{t}))^2 (v_0^\pm(\tilde{t}))^2 \, d\tilde{t}.
\]

\section*{Acknowledgements}
Marco Polo Garc\'{\i}a-Rivera
was partially supported by a SECIHTI Mexico postgraduate fellowship No. 905424.
M. Alvarez-Ram\'{\i}rez is supported by the 2026 Special Program for Teaching and Research Project Funding at CBI-UAMI.

\bibliographystyle{plain}
\bibliography{refmarco_ice2026}
\end{document}